\documentclass[11pt]{article}

\usepackage[cmex10]{amsmath}
\usepackage{amssymb,amsfonts,amsthm}
\usepackage{mathrsfs}
\usepackage{graphicx}
\usepackage{caption,subcaption}
\usepackage{colortbl}
\usepackage{multirow,multicol,diagbox,slashbox}
\usepackage{enumerate}
\usepackage{lineno}
\usepackage[numbers,sort&compress]{natbib}
\usepackage{multirow,multicol,diagbox}  
\usepackage{authblk}

\usepackage[hidelinks]{hyperref}

\newtheorem{definition}{Definition}
\newtheorem{theoremx}{Theorem}
\newtheorem{lemmax}{Lemma}

\newtheorem{remark}{Remark}

\newtheorem{example}{Example}
\newtheorem{assumption}{Assumption}
\renewcommand\thefootnote{\fnsymbol{footnote}} 

\title{Semiglobal Input-Delay Tolerance Algorithm for Distributed Nonconvex Optimization\\
of Networked Nonlinear Systems%
\thanks{This manuscript is an extended version of our paper published in \emph{IEEE Transactions on Automatic Control}, {2026}, doi:\textbf{10.1109/TAC.2026.3702930}. \copyright~\textbf{2026} IEEE. Personal use is permitted. For other uses, permission must be obtained from IEEE.}%
}

\author{
Jing-Zhe Xu\textsuperscript{1},
Zhi-Wei Liu\textsuperscript{1},
Ming-Feng Ge\textsuperscript{2},
Yan-Wu Wang\textsuperscript{1},
and Dinxin He\textsuperscript{1}
}

\date{} 

\begin{document}
\maketitle

\renewcommand\thefootnote{\arabic{footnote}} 
\setcounter{footnote}{0}

\footnotetext[1]{School of Artificial Intelligence and Automation, Huazhong University of Science and Technology, Wuhan 430074, China (e-mail: jzxu@hust.edu.cn; zwliu@hust.edu.cn; wangyw@hust.edu.cn; hedingxin@hust.edu.cn).}

\footnotetext[2]{School of Mechanical Engineering and Electronic Information, China University of Geosciences, Wuhan 430074, China (e-mail: gemf@cug.edu.cn).}

\footnotetext[3]{\textbf{Acknowledgment:} This work was supported by the National Natural Science Foundation of China under Grants U24A20268, 624B2055, 62373162, and 62473349.}

\begin{abstract}
This paper studies a class of distributed optimization problems in networked nonlinear systems (NNSs) subject to input delays and consensus constraints.
It introduces input-delay tolerant semiglobal convergence (IDTSC), meaning that for any prescribed compact initial set there exists an admissible delay bound under which the optimal solution is computed within consensus constraints and all node states converge to the solution.
Building on a hierarchical design and input-to-state stability analysis,
a new semiglobal input-delay tolerant (SIDT) algorithm is developed that practically achieves IDTSC for distributed optimization under the coupling between input delays and nonlinear dynamics.
Further, by relaxing strict convexity requirements through the Polyak-{\L}ojasiewicz condition,
the SIDT algorithm broadens its applicability to nonconvex optimization.
Finally, numerical experiments corroborate the theory on NNSs with input delays.
\end{abstract}


\section{Introduction}
Distributed optimization has attracted sustained attention in networked systems because of its decomposability, data locality, and robustness to communication imperfections \cite{LiT2008TAC,NediA2018,TatarenkoT2017,WangY2023,WangQ2019,HouskaB2016,XuL2023,Firouzbahrami2022-2}.
Most of these works now exist along two main lines: discrete-time strategies  \cite{TatarenkoT2017,WangY2023,WangQ2019} and continuous-time strategies \cite{HouskaB2016,XuL2023,Firouzbahrami2022-2}, to find the solution of the optimization problem.
In most continuous-time studies, the ``state" typically represents an algorithmic iterate rather than a physical plant state, and node dynamics are either omitted or idealized as single-integrator models (effectively $\dot x_i = u_i$) \cite{HouskaB2016,XuL2023,Firouzbahrami2022-2}.
As a consequence, these results only focus on solving the optimization computation problem itself essentially, not the joint problem of regulating a dynamical network while solving the distributed optimization task.

As the field has matured, distributed optimization has moved from purely algorithmic studies to practical deployments in sensor networks \cite{Rabbat2004,ChenJ2016}, machine learning \cite{Nedic2020}, smart grids \cite{AdikaC2014,YuW2018} and robotics networks \cite{Riviere2021,LiangCD2023}.
In these settings the decision variables are the physical states of nodes rather than abstract iterates, and the dynamics of the node are richer than a single integrator and often nonlinear \cite{Rabbat2004,ChenJ2016,Nedic2020,AdikaC2014,YuW2018,Riviere2021,LiangCD2023}.
Thus, algorithms that merely compute an optimizer \cite{HouskaB2016,XuL2023,Firouzbahrami2022-2} are not directly applicable since the design must both solve the distributed optimization problem and regulate the nonlinear dynamics so that the states converge to the optimizer.
To this end, recent works therefore blend control with optimization \cite{YuW2018,Riviere2021,LiangCD2023,LiuT2021,WangX2015}.
For example, hierarchical distributed controllers have been developed for fleets with Euler-Lagrange dynamics, including unmanned surface vehicles \cite{LiangCD2023}, and small gain methods are used to achieve optimal output consensus in uncertain nonlinear multi-agent systems \cite{LiuT2021}.

However, most of these optimization control designs are developed under delay free assumptions \cite{Riviere2021,LiangCD2023,Rabbat2004,ChenJ2016,Nedic2020,AdikaC2014,YuW2018,LiuT2021,WangX2015}.
In above application domains (e.g., robotics networks) that apply these methods, delays are difficult to avoid \cite{Abadia2021,JinM2017,MuslehAS2018}.
Specifically, hardware limits and environmental disturbances introduce latency in sensing, computation, communication, and actuation \cite{Abadia2021}.
When delays interact with nonlinear dynamics, they corrupt gradient feedback and create residual terms that can destroy forward invariance and, even when arbitrarily small, undermine stability and prevent convergence to the optimal solution \cite{LinWei2020}.
Previous delay handling optimization methods use Lyapunov Krasovskii functionals and are effective only for linear systems with delays \cite{YangS2016,WangD2018}.
Moreover, optimization methods based on switching communication graphs face related limits \cite{LinPAuto2016}, as they offer no general stability guarantee for nonlinear plants.
Therefore, distributed optimization control for networked nonlinear systems with delays remains largely gap.

Furthermore, many practical applications of distributed optimization, including multirobot motion planning \cite{YaziciA2013,LiangCD2023} and economic dispatch in power grids \cite{YiP2015}, necessitate strict adherence to state constraints \cite{YaziciA2013,GharesifardB2013,YiP2015,ZengX2016,LiuQ2017,WangXF2020,LiangCD2023}.
These state constraints are intrinsic to real systems, often arising from physical and safety limits, task-level coordination that enforces state consistency, and conservation requirements at the network scale \cite{YaziciA2013,YiP2015}.
Nevertheless, designing algorithms that minimize global cost functions within the required constraints is a formidable challenge.
To address this, consensus-based optimization approaches utilizing nonsmooth Lyapunov functions have been proposed \cite{LiangCD2023,GharesifardB2013,ZengX2016,LiuQ2017,WangXF2020}.
Note that these methods remain inadequate for systems with input delays, as they rely on optimization in delay-free environments that consider only current states and control inputs.
The presence of input delays necessitates the incorporation of historical state information, thereby complicating the enforcement of real-time constraints.
This fundamental alteration in the optimization problem's structure underscores the need for more robust algorithms capable of handling input delays in constrained distributed optimization.

Building on the insights from previous discussions, this paper investigates a class of distributed convex and nonconvex optimizations for networked nonlinear systems (NNSs) subject to input delays and consensus constraints.
To this end, a novel optimization control algorithm is proposed to compute the optimal solution within consensus constraints and, at the same time, regulate the node states to approach the computed optimal solution.
The key contributions of this paper are summarized as follows:
\begin{enumerate}
  \item A novel concept of input-delay tolerant semiglobal convergence (IDTSC) is provided to cope with the fact that unknown input delays together with nonlinear node dynamics can break the solvability of the distributed optimization problem and prevent the network state from converging.
      Compared with linear-delay optimization schemes \cite{LinPAuto2016,YangS2016,WangD2018}, IDTSC concept follows a semiglobal principle that links the radius of the initial condition set to an admissible delay margin, thereby offering implementable guarantees for distributed optimization over nonlinear nodes.
  \item Unlike existing distributed constrained optimization methods without considering time delays \cite{GharesifardB2013,ZengX2016,LiuQ2017,WangXF2020}, the proposed algorithm couples a hierarchical structure with input-to-state stability (ISS) theory guaranteeing that the states return to and stay in the constraint set and are driven toward the constrained optimizer under the input delay.
  Importantly, the algorithm is delay-independent, implying that its gains do not depend on the delay and require no delay-based retuning.
  \item 
  The strong-convexity requirement is relaxed to the Polyak–{\L}ojasiewicz (P–{\L}) condition, under which the proposed design still guarantees input-delay tolerant practical semiglobal convergence for distributed optimization of NNSs with input delays and consensus constraints.
\end{enumerate}

The paper is organized as follows:
Section \uppercase\expandafter{\romannumeral2} introduces preliminaries and formulates the problem.
Section \uppercase\expandafter{\romannumeral3} presents the SIDT algorithm and the theoretical results.
Section \uppercase\expandafter{\romannumeral4} provides numerical validation, and section \uppercase\expandafter{\romannumeral5} draws a conclusion with future directions.

\section{Problem Formulation And Preliminary}

\subsection{Notations and Definitions}
Notations:
Let \(\mathbb{R}^n\) denote the \(n\)-dimensional real vector space, and \(\mathbb{R}^{n \times n}\) denotes the space of \(n \times n\) real matrices.
The absolute value of a variable \( x \) is denoted by \( |x| \), and the function \(\textrm{sig}(x)^a\) is defined as \(\textrm{sign}(x) |x|^a\).
Additionally, \(\textbf{I}_n\) represents \(n \times n\) diagonal matrix with ones on the main diagonal.
$\mathbf{1}_n$ and \(\mathbf{0}_n\) denote the $n$-dimensional all-ones and all-zero column vector, respectively.
\(\nabla f(x)\) and \(\nabla^2 f(x)\) represent the gradient vector and Hessian matrix of a scalar field \( f(x) \), respectively.
Let $\otimes$ denote the Kronecker product.

The following stability concept is foundational for in this paper.

\begin{definition}\textbf{(Global asymptotic and local exponential stability)}\label{definition_GALES}
Consider a nonlinear system:
\begin{align}\label{definition_system1}
\dot{x} = f(x, u),
\end{align}
where \( x \in \mathbb R^n \) is the state, and \( u \in \mathbb R^m  \) is the feedback control law.
This system (\ref{definition_system1}) is said to be globally asymptotically and locally exponentially (GALE) stabilizable through state feedback if there exists a smooth feedback control law \( u = \xi(x) \),
where \(\xi : \mathbb{R}^n \rightarrow \mathbb{R}^m\) is a smooth function satisfying \(\xi(0) = 0\), such that the closed-loop system:
\[
\dot{x} = f(x, \xi(x)) \triangleq \bar{f}(x),
\]
exhibits global asymptotic stability and local exponential stability at the equilibrium point \( x = 0 \).
\end{definition}

\subsection{Communication graph}
Consider an undirected graph \({\mathcal{G}} = ({\mathcal{V}}, {\mathcal{E}})\), where \({\mathcal{V}}\) represents a set of \(N\) nodes, and \({\mathcal{E}} \subseteq {\mathcal{V}} \times {\mathcal{V}}\) denotes the edges between them.
For each node \(i\), the neighborhood set is defined as \(\mathcal{N}_i = \{ j \in {\mathcal{V}} \mid (i, j) \in {\mathcal{E}} \}\).
The connectivity of the graph is represented by the adjacency matrix \(\mathcal{A}\), where \(a_{ij} > 0\) if nodes \(i\) and \(j\) are connected (i.e., \((i, j) \in {\mathcal{E}}\)), and \(a_{ij} = 0\) otherwise.
The degree matrix \(\mathcal{D}\) is a diagonal matrix, with the \(i\)-th diagonal entry given by \(d_i = \sum_{j \neq i} a_{ij}\), representing the degree of node $i$.
The Laplacian matrix \({\mathcal{L}}\) of the graph is then defined as ${\mathcal{L}} = \mathcal{D} - \mathcal{A}$.

\subsection{Problem Description}
Consider a NNS of $n$ nodes interacting over a graph ${\mathcal G}=(\mathcal V,\mathcal E)$.
Each node $i \in \mathcal V$ has scalar state $x_i(t)\in\mathbb R^m$ and the input is $u_i(t)\in\mathbb R^m$.
The dynamics are governed by:
\begin{align}\label{system}
\dot{x}_i(t)=  f_i(x_i(t)) + g_i(x_i(t))u_i(t-d), \quad i = 1, \ldots, n,
\end{align}
where $d$ denotes input delay,
smooth maps $f_i: \mathbb R^m \to \mathbb R^m$ and $g_i:\mathbb R^m \to \mathbb R$ are smooth,
and in particular $f_i(\mathbf{0}_m) = \mathbf{0}_m$.
In what follows, every occurrence of a delayed argument $y(t-d)$ is to be interpreted as $y$ evaluated at a delayed time stamp, that is, $y(t-d)$ denotes the most recent available sample of $y$ at time $t$,
where $y$ may denote any signal constructed from the states and exchanged variables (e.g., $x_i$, $u_i$).

Then, each node $i$ is endowed with a local cost function $\phi_i:\mathbb R^m \to \mathbb R$.
Let $x := [x_1,\dots,x_n]^T \in\mathbb R^{nm}$ and define the aggregate objective
$\Phi(x) := \sum_{i=1}^n \phi_i(x_i)$.
The goal is to minimize the aggregate cost under consensus constraints:
\begin{align}\label{Cx}
&\min_{x \in \mathbb{R}^{nm}} \Phi (x) = \sum_{i=1}^{n} \phi_{i}(x_i), \notag\\
&\text{subject to} \quad x_i = x_j, \quad \forall i,j \in {\mathcal{V}},
\end{align}
where the constraint $x_i=x_j$, $\forall i,j \in {\mathcal{V}}$.
Denote the optimal set by:
\begin{align*}
	\mathcal X^* := \arg\min_{x\in\mathbb R^{nm}, x_i=x_j, \forall i,j \in {\mathcal{V}} }\ \Phi(x),
\end{align*}
assumed nonempty, and the optimal value by $\Phi^* := \min_{x\in\mathbb R^{nm},\ x_i=x_j}\Phi(x)>-\infty$.
The control objective is to design distributed inputs $u_i$ such that the NNS \eqref{system} drive the state $x$ to the optimal consensual solution $x^* \in \mathcal X^*$,
where $x^* = [x_1^*,\ldots,x_n^*]^T$ and $x_i^* = x_j^*$, $\forall i ,j \in {\mathcal V}$.
Several needed assumptions are provided as follows.



\begin{assumption}
For each node $i$, the scalar gain $g_i:\mathbb R^m \to \mathbb R$ is $C^{1}$ and bounded away from zero. Namely, there exist constants $0<\underline g_i\le \overline g_i<\infty$ such that:
$$
\underline g_i \le |g_i(x)| \le \overline g_i.
$$
\end{assumption}

\begin{assumption}
Each local cost function \(\phi_i(x)\) is twice continuously differentiable with respect to \(x\) and globally \(L_{\phi_i}\)-Lipschitz continuous gradient. For \(x_{1}, x_{2} \in \mathbb{R}^m\), the following holds:
\[
|\nabla \phi_i(x_{1}) - \nabla \phi_i(x_{2})| \leq L_{\phi_i} ||x_{1} - x_{2}||.
\]
\end{assumption}

\begin{assumption} (Strong convexity condition)
Each local cost function \(\phi_i(x)\) is strongly convex with parameter \(m_{\phi_i} > 0\). Specifically, for any \(x_{1}, x_{2} \in \mathbb{R}^m\) and \(\theta \in [0, 1]\), the following condition holds:
\begin{align*}
\phi_i(\theta x_{1} + (1-\theta)x_{2}) \leq & \theta \phi_i(x_{1}) + (1-\theta)\phi_i(x_{2})  - \frac{m_{\phi_i}}{2} \theta (1-\theta) ||x_{1} - x_{2}||^2.
\end{align*}
\end{assumption}

Building on the challenges outlined earlier, we now formalize the concept of input-delay tolerant semiglobal convergence to address distributed optimization problems under input delays.

%
%

\begin{definition}\textbf{(Input-Delay Tolerant Semiglobal Convergence (IDTSC) in Distributed Optimization)}\label{definition_C}
Consider the distributed constrained optimization problem (\ref{Cx}) implemented on a networked nonlinear system (\ref{system}) with input delays. Let:
\[
u_i(t) = \xi_i(x_i(t)),
\]
for each node \(i\), where \(\xi_i: \mathbb{R}^m \to \mathbb{R}^m\) is a globally stabilizing control law. Denote by \(x_i^* \in \mathbb{R}^{m}\) the optimal solution of the optimization problem (\ref{Cx}). For any prescribed constant \(r > 0\), define the initial condition ball:
\[
\mathcal{B}_r(x_i^*) \triangleq \big\{ \varphi \in C([-d, 0],\mathbb{R}) : \|\varphi - x_i^*\| \le r \big\},
\]
where the norm is given by \(\|\varphi\| = \sup_{s \in [-d, 0]} \|\varphi(s)\|\). We say that the closed-loop system exhibits \emph{input-delay tolerant semiglobal convergence (IDTSC)} in distributed optimization if there exists a maximal allowable delay \(\bar{d} = \bar{d}(r) > 0\) such that, for all delay values \(d \in (0,\bar{d}]\), the following properties hold:
\begin{enumerate}
    \item \emph{Semiglobal Attractivity}: For every \(i\in \mathcal{V}\) and any initial condition \( x_i(s) \in \mathcal{B}_r(x_i^*) \), the trajectory \(x_i(t)\) of the closed-loop system,
    \begin{align}\label{closed-system1}
    \dot{x}_i(t) = f_i(x_i(t)) + g_i(x_i(t))\,\xi_i(x_i(t-d)),
    \end{align}
    satisfies
    \[
    \lim_{t \to \infty} x_i(t) = x_i^*.
    \]
    \item \emph{Local Stability}: The closed-loop system is locally asymptotically stable at the optimal solution.
\end{enumerate}
\end{definition}

To extend the scope of IDTSC beyond convex optimization, we refine the Assumption 3 on local cost functions to accommodate a broader class of nonconvex problems by employing the Polyak-{\L}ojasiewicz (P-{\L}) condition.

\begin{assumption}
Each local cost function \(\phi_i(x_i)\) satisfies radially unbounded and the P-{\L} condition:
\begin{equation}\label{assumption3}
\frac{1}{2} \|\nabla \phi_i(x_i) \|^2 \geq \mu_{\phi_i} (\phi_i(x_i) - \phi_i(x_i^*)),
\end{equation}
where \(\mu_{\phi_i}\) is a positive constant, and \(x_i^*\) is an optimal solution of the optimization problem (\ref{Cx}).
\end{assumption}

\begin{remark}
The P-{\L} condition in Assumption 4 provides a less restrictive alternative to strong convexity.
It relaxes the requirements compared to assumptions like essential strong convexity \cite{LiuJ2014}, weak strong convexity \cite{NecoaraI2019}, or the restricted secant inequality \cite{RiedmillerM1992}.
This flexibility allows the proposed framework to handle a broader class of nonconvex optimization problems \cite{Karimi2016,PovedaJIACC2020}, significantly expanding its practical applicability.
\end{remark}

\subsection{Preliminary lemmas}
To support the theoretical analysis of the proposed algorithm, we summarize several essential lemmas that establish foundational properties and facilitate the stability and convergence proofs.

\begin{lemmax}\label{lemma1}
\cite{RudinW1976}
({Cauchy-Schwarz inequality})
Let \( x = [x_1, x_2, \dots, x_n]^T \) and \( y = [y_1, y_2, \dots, y_n]^T \) be vectors in \( \mathbb{R}^n \). Then, the following inequality holds:
\[
\sum_{i=1}^n |x_i| |y_i| \leq \|x\| \|y\|,
\]
where \( \|x\| = \left( \sum_{i=1}^n |x_i|^2 \right)^{1/2} \) and \( \|y\| = \left( \sum_{i=1}^n |y_i|^2 \right)^{1/2} \) are the Euclidean norms of \( x \) and \( y \), respectively.
\end{lemmax}

%
%
%
%
%
%
%


\begin{lemmax}\label{lemma2}
{\cite{LinWei2020}}
Let \( f : \mathbb{R}^n \to \mathbb{R}^n \) be a smooth mapping with \( f(0) = 0 \). Then, there exists a smooth, non-negative scalar function \( \Psi : \mathbb{R}^n \times \mathbb{R}^n \to \mathbb{R} \) such that for all \( x, y \in \mathbb{R}^n \),
\[
\| f(x) - f(y) \| \leq \Psi(x, y) \| x - y \|.
\]
\end{lemmax}

\begin{lemmax}\label{lemma3}
{\cite{LinWei2020}}
Let \( f : \mathbb{R}^n \times \mathbb{R}^m \to \mathbb{R} \) be a smooth function. Then, there exist smooth scalar functions \( a(x) \geq 0 \), \( b(y) \geq 0 \), \( c(x) \geq 1 \), and \( d(y) \geq 1 \) for all \( x \in \mathbb{R}^n \) and \( y \in \mathbb{R}^m \), such that
\[
| f(x, y) | \leq a(x) + b(y) \quad  \textrm{and} \quad | f(x, y) | \leq c(x) d(y).
\]
\end{lemmax}

\begin{lemmax}\label{lemma5}
\cite{Karimi2016, XuL2023,PovedaJIACC2020}
Given Assumption 4 (P-{\L} condition), the distributed optimization problem (\ref{Cx}) associated with the NNS (\ref{system}) has at least one optimal solution \( x^* \). At this optimal point, the gradient vector field satisfies \( \nabla \Phi (x^*) = 0 \).
\end{lemmax}

\begin{lemmax}\label{lemma6}
\cite{GuK2003} ({Razumikhin theorem})
Consider a nonlinear system with time-delay described by:
\[
\dot{x}(t) = f(x_t),
\]
where \( f : C([-d, 0], \mathbb{R}^n) \rightarrow \mathbb{R}^n \) is a locally Lipschitz continuous function with \( f(0) = 0 \), and the initial condition \( x(\theta) = \varphi(\theta) \) is continuous for \( \theta \in [-d, 0] \).

The system is globally asymptotically stable if there exist:
\begin{enumerate}
    \item A first-order continuously derivable (i.e., \( C^1 \)) function \( V : \mathbb{R}^n \rightarrow \mathbb{R} \).
    \item A continuous (i.e., \( C^0 \)) function \( W : \mathbb{R}^n \rightarrow \mathbb{R} \), both positive definite and proper.
    \item A \( C^0 \) non-decreasing function \( p(s) \) such that \( p(s) > s \) for \( s > 0 \).
\end{enumerate}
such that for all \( \theta \in [-d, 0] \), the following condition holds:
\[
\dot{V}(x(t)) \leq -W(x(t)) \quad \text{whenever} \quad V(x(t + \theta)) \leq p(V(x(t))).
\]
\end{lemmax}

\begin{lemmax}\label{lemma7}
\cite{SontagE2013} (Input-to-State Stability)
Consider the nonlinear dynamical system
\[
\dot{x}(t) = f(x(t), u(t)),
\]
where \( x(t) \in \mathbb{R}^n \) denotes the state vector, \( u(t) \in \mathbb{R}^m \) represents the input vector, and the function \( f: \mathbb{R}^n \times \mathbb{R}^m \to \mathbb{R}^n \) is locally Lipschitz continuous in its arguments.

Assume the existence of a continuously differentiable Lyapunov function \( V: \mathbb{R}^n \to \mathbb{R} \), a class \( \mathcal{K} \) function \( \gamma \), and a positive constant \( \alpha > 0 \) such that for all \( x \in \mathbb{R}^n \) and \( u \in \mathbb{R}^m \), the following conditions are satisfied:
\begin{enumerate}
    \item[a)] \( \alpha_1(\|x\|) \leq V(x) \leq \alpha_2(\|x\|) \), where \( \alpha_1, \alpha_2 \in \mathcal{K}_\infty \).
    \item[b)] \( \nabla V(x) \cdot f(x, u) \leq -\alpha V(x) + \gamma(\|u\|) \).
\end{enumerate}

Under these conditions, the system is input-to-state stable (ISS) with respect to the input \( u(t) \). Specifically, there exist functions \( \beta \in \mathcal{KL} \) and \( \tilde{\gamma} \in \mathcal{K} \) such that for any initial state \( x(0) \) and input \( u(t) \), the following bound holds for all \( t \geq 0 \):
\begin{align}\label{ISS-Lemma}
\|x(t)\| \leq \beta(\|x(0)\|,t) + \tilde{\gamma}\left(\sup_{0 \leq \tau \leq t} \|u(\tau)\|\right), \quad \forall t \geq 0.
\end{align}
\end{lemmax}


\section{Main Results}
This section presents the main contributions of the paper, focusing on the development of a novel SIDT algorithm.
The algorithm addresses both distributed convex and nonconvex optimization problems (\ref{Cx}) in NNS (\ref{system}) with input delays.

\subsection{SIDT algorithm design}
To achieve IDTSC, we propose a novel SIDT algorithm designed to solve distributed optimization problems (\ref{Cx}) in delay-affected NNSs (\ref{system}).
The control input for each node is formulated as:
\begin{align}\label{SIDT}
u_i (t - d)
=& g_i(x_i(t - d))^{-1}(\bar{u}_i(t - d) - f_i(x_i(t - d))),
\end{align}
where the system terms $f_i$ and $g_i$ are known, and ${\bar u}_i(t-d)$ is an auxiliary control term given by:
\begin{align*}
\bar{u}_i (t - d) =& \vartheta \sum_{j = 1}^n a_{ij} (x_j(t - d) - x_i(t - d))\notag\\
& + \varepsilon \sum_{j \in {\mathcal N}_i} \big[ \wp_{ij}(t - d) \cdot \textrm{sign} (x_j(t - d) - x_i(t - d))\big] + \bar u_{\eta,i}(t - d),
\end{align*}
and
\begin{align*}
\bar u_{\eta,i}(t - d) =  - k_{0} \eta_i(x_i (t - d)).
\end{align*}
The terms are specified as follows: $ \vartheta  > 0$ is a gain parameter,
$\wp_{ij}(t - d) = \|\bar u_{\eta,i}(t - d)\| + \|\bar u_{\eta,j}(t - d)\|$,
$\eta_i(x_i (t - d)) = \nabla \phi_i(x_i(t - d))$,
$k_{0} > 0$ is a scaling parameter for the gradient-based term,
and $\varepsilon$ is a control gain defined later.
In scenarios where there are no input delays \( d = 0 \), the SIDT algorithm (\ref{SIDT}) simplifies to:
\begin{align}\label{SIDT_t}
u_i (t)
=  g_i(x_i (t))^{-1}(-f_i(x_i (t)) + {\bar u}_i(t)),
\end{align}
where
\begin{align*}
{\bar u}_i (t) = & \vartheta \sum_{j = 1}^n a_{ij}(x_j (t) - x_i (t)) + \varepsilon \sum_{j \in {\mathcal N}_i} \big[\wp_{ij}(t) \cdot \textrm{sign}(x_j(t) - x_i(t))\big] + \bar u_{\eta,i}, \\
\bar u_{\eta,i} = &- k_{0} \eta_i(x_i (t)),\\
\end{align*}
and
\begin{align*}
\eta_i(x_i (t)) = & \nabla \phi_i (x_i (t)).
\end{align*}

\subsection{SIDT algorithm for distributed convex optimization problem}
In this subsection, we analyze the application of the SIDT algorithm (\ref{SIDT}) to solve the distributed convex optimization problem (\ref{Cx}) in NNSs (\ref{system}) with input delays.
The analysis leads to the following result:

\begin{theoremx}\label{T1}
Assuming Assumptions 1-3 are satisfied, if $\varepsilon \geq 2n$ holds, then the application of the SIDT algorithm (\ref{SIDT}) to the distributed optimization problem (\ref{Cx}) of the NNS (\ref{system}) with input delays guarantees that the system exhibits practical IDTSC.
\end{theoremx}

\begin{proof}
The proof proceeds in two main steps.

\textbf{Step 1:} Establish that the NNS (\ref{system}) without the input delay $d$, utilizing the SIDT algorithm (\ref{SIDT_t}), achieves global asymptotic and local exponential stability at the optimal solution of the distributed optimization problem (\ref{Cx}).
\textbf{Step 2:} Demonstrate that the SIDT algorithm (\ref{SIDT}) ensures practical IDTSC in the presence of input delay $d$.

\textbf{\textrm{Step 1 (Stability without input delay):}}
Consider the NNS (\ref{system}) without input delay, formulated as:
\begin{align}\label{system_nd}
\dot x_{i} (t) = f_{i}(x_i(t)) + g_i (x_i(t)) u_i(t), \quad i = 1,\ldots,n.
\end{align}
Substituting the control input $u_i(t)$ from (\ref{SIDT_t}) into (\ref{system_nd}):
\begin{align}\label{closed-sys}
\dot x_{i} (t) = f_{i}(x_i(t)) + g_i (x_i(t)) \xi_i(x_i(t)).
\end{align}
Construct the following Lyapunov function candidate as follows:
\begin{align}\label{V_t1}
V_{pre} =\frac{1}{2} \sum_{i = 1}^n e_{x_i}^T e_{x_i},
\end{align}
where $e_{x_i} = x_i - \frac{1}{n} \sum_{j = 1}^n x_j$.
The derivative of (\ref{V_t1}) along the trajectories of (\ref{closed-sys}) yields:
\begin{align}\label{V_t12}
\dot V_{pre} =& \sum_{i = 1}^n e_{x_i}^T \cdot \dot e_{x_i} \notag\\
= &\sum_{i = 1}^n e_{x_i}^T (  \dot x_i - \frac{1}{n} \sum_{j = 1}^n \dot x_j).
\end{align}
Since $\sum_{i = 1}^n e_{x_i} = 0$,
it follows that:
$$- \frac{1}{n} \sum_{i = 1}^n e_{x_i} \cdot (\sum_{j = 1}^n \dot x_j) = 0,$$
implying that:
\begin{align}\label{V_t13}
\dot V_{pre} =& \sum_{i = 1}^n e_{x_i}^T \dot x_i  \notag\\
= & -  \vartheta  \sum_{i = 1}^n\sum_{j = 1}^n  e_{x_i}^T a_{ij}(x_i (t) - x_j (t)) - \varepsilon \sum_{i = 1}^n \sum_{j \in {\mathcal N}_i}  e_{x_i}^T \wp_{ij}(t) \textrm{sign}(x_i(t) - x_j(t)) \notag\\
& + \sum_{i = 1}^n e_{x_i}^T \bar u_{\eta,i}.
\end{align}

To analyze the term $ I_1 = -  \vartheta  \sum_{i = 1}^n e_{x_i}^T \sum_{j = 1}^n a_{ij}(x_i (t) - x_j (t)) $ in (\ref{V_t13}),
one has:
\begin{align}\label{Iterm_1}
I_1
=& - \frac{\vartheta}{2} \sum_{i = 1}^n \sum_{j = 1}^n a_{ij} e_{x_i}^T (x_i - x_j) - \frac{\vartheta}{2} \sum_{i = 1}^n \sum_{j = 1}^n a_{ij} e_{x_j}^T (x_j - x_i) \notag\\
=& - \frac{\vartheta}{2} \sum_{i = 1}^n \sum_{j = 1}^n a_{ij} (x_i - x_j)^T (x_i - x_j)\notag\\
=& - \vartheta e_x^T {\bar {\mathcal L}} e_x \notag\\
\leq & -2\vartheta \lambda_2({\bar {\mathcal L}}) V_{pre},
\end{align}
where $e_x = [e_{x_1},e_{x_2},\ldots,e_{x_n}]^T$, ${\bar {\mathcal L}} = {\mathcal L} \otimes \textbf{I}_m$, and \(\lambda_{2}(\bar {\mathcal L})\) denotes the smallest non-zero eigenvalue of $\bar {\mathcal L}$.

For $ I_2 = - \varepsilon \sum_{i = 1}^n\sum_{j \in {\mathcal N}_i}  e_{x_i}^T \wp_{ij}(t) (x_i(t) - x_j(t)) $,
it yields that:
\begin{align}\label{Iterm_2}
I_2 =& - \frac{\varepsilon}{2} \sum_{i = 1}^n \sum_{j \in {\mathcal N}_i} e_{x_i}^T \wp_{ij} \textrm{sign} (x_i - x_j) - \frac{\varepsilon}{2} \sum_{i = 1}^n \sum_{j \in {\mathcal N}_i} e_{x_j}^T \wp_{ji} \textrm{sign} (x_j - x_i) \notag\\
=& - \frac{\varepsilon}{2} \sum_{i = 1}^n \sum_{j \in {\mathcal N}_i} \wp_{ij} (x_i - x_j)^T \textrm{sign} (x_i - x_j)\notag\\
\leq& - \frac{\varepsilon}{2} \sum_{i = 1}^n \sum_{j \in {\mathcal N}_i} \|x_i - x_j\| \| \bar u_{\eta,i} \|.
\end{align}

To analyze the term $ I_3 = - \sum_{i = 1}^n e_{x_i}^T \bar u_{\eta,i}$,
note that:
$$\max_{i,j \in {\mathcal V}} \| x_{i}(t) - x_{j}(t) \| \leq \sum_{i = 1}^n \sum_{j \in {\mathcal N}_i} \| x_{i} - x_{j} \|,$$
which implies that:
\begin{align}\label{Iterm_3}
I_3 \leq& \sum_{i = 1}^n \|e_{x_i}\| \|\bar u_{\eta,i}\| \notag\\
\leq& \sum_{i = 1}^n \left(\sum_{i = 1}^n \sum_{j \in {\mathcal N}_i} \| x_{i} - x_{j} \| \right)\|\bar u_{\eta,i}\| \notag\\
\leq& n \sum_{i = 1}^n \sum_{j \in {\mathcal N}_i} \| x_{i} - x_{j} \| \|\bar u_{\eta,i}\|.
\end{align}

According to the condition $\varepsilon \geq 2n$,
summing up equations (\ref{Iterm_1})-(\ref{Iterm_3}) and substituting them into (\ref{V_t13}) yields:
\begin{align}\label{V_t14}
\dot V_{pre} \leq& -2\vartheta \lambda_2(\bar {\mathcal L}) V_{pre}.
\end{align}
From \eqref{V_t14},
together with the quadratic bounds $c_1\|e_x(t)\|^2 \le V_{pre}(t) \le c_2\|e_x(t)\|^2$ (for some $c_1$, $c_2>0$), we obtain that $e_{x_i}$ globally converges exponentially to origin,
namely, $\lim_{t \to \infty} \| x_i(t) - \frac{1}{n} \sum_{j = 1}^n x_j \| = 0$, $\forall i \in {\mathcal V}$.

Since $\lim_{t \to \infty} \| x_i(t) - \frac{1}{n} \sum_{j = 1}^n x_j \| = 0$, $\forall i \in {\mathcal V}$, it follows that:
\begin{equation}\label{V_t15}
x_i(t) = x_j(t), \quad \forall i,j \in {\mathcal V},\quad \text{as } ~ t \to \infty.
\end{equation}
Equivalently, the state converges to the consensus subspace:
$$
\mathcal C := \{x \in \mathbb R^{nm}: x_i = x_j, \forall i,j \in \mathcal V\}.
$$
Let $e_x(t) = Px(t)$, where $P = P_n \otimes \textbf{I}_m$ and $P_n = \textbf{I}_n-\tfrac{1}{n}{\mathbf 1}_n {\mathbf 1}_n^T$.
Note that the disagreement dynamics evolve in the orthogonal complement of $\mathrm{span} \{ {\mathbf 1}_n \}\otimes \mathbb R^m$.
In particular,
\[
P(\mathbf 1_n\otimes y)=0,~ \forall y\in\mathbb R^m,
\quad
(\mathbf 1_n^T \otimes \textbf{I}_m) ( \mathcal L \otimes \textbf{I}_m)=0,
\]
so any auxiliary input that acts through the consensus channel (i.e., takes values in
$\mathrm{span}\{\mathbf 1_n\}\otimes \mathbb R^m$, including the term
$\bar u_{\eta,i}(t-d)$ as designed in \eqref{SIDT}-\eqref{SIDT_t}) is annihilated by $P$ and does not enter the $e_x$-subsystem.
Therefore, \eqref{V_t14} guarantees exponential decay of the disagreement independently of the specific variation of $\bar u_{\eta,i}(t-d)$, and the closed-loop trajectories necessarily enter (and remain in) $\mathcal C$.

Then, select a new Lyapunov function candidate as follows:
\begin{align}\label{V_nd1}
V_{1} = \Phi (x(t)) - \Phi (x^*(t)),
\end{align}
where $x^*(t) \in \mathbb R^{nm}$ is the optimal solution of the distributed optimization problem (\ref{Cx}).
According to Assumption 2, the aggregate cost function $\Phi (x(t))$ is $L_{\Phi}$-Lipschitz continuous in its gradient:
\begin{align}\label{LpC}
\| \nabla \Phi (x) - \nabla \Phi (y) \| \leq L_{\Phi} \| x - y \| , \quad \forall x,y \in \mathbb R^{nm},
\end{align}
where $L_{\Phi} = \max_{i =1,\ldots,n} \{ L_{\phi_i}\}$, $\forall i \in {\mathcal V}$.

Using the first-order Taylor expansion for $\Phi$:
\begin{align}\label{Lp-1}
\Phi (y) \leq& \Phi (x) + \nabla \Phi (x)^T (y - x) + \int^1_0 (\nabla \Phi (x + t(y - x)) - \nabla \Phi (x))^T (y - x) dt.
\end{align}
The $L_{\Phi}$-Lipschitz continuity of $\nabla \Phi \left( x \right)$ (\ref{LpC}) allows us to bound the integral term:
\begin{align}\label{Lp-2}
\int^1_0 \| \nabla \Phi (x + t(y - x)) - \nabla \Phi (x) \| dt \leq&  \int^1_0 L_{\Phi} t \| y - x\| dt \notag\\
 = & \frac{L_{\Phi}}{2} \| y - x\|^2.
\end{align}
Substituting back, we get:
\begin{align}\label{Lp-3}
\int^1_0 (\nabla \Phi (x + t(y - x)) - \nabla \Phi (x))^T (y - x) dt \leq& \frac{L_{\Phi}}{2} \| y - x\|^2.
\end{align}
Substituting (\ref{Lp-3}) into (\ref{Lp-1}),
it obtains that:
\begin{align}\label{Lp-4}
\Phi (y) \leq& \Phi (x) + \nabla \Phi (x)^T (y - x) + \frac{L_{\Phi}}{2} \| y - x\|^2 .
\end{align}
Setting $x = x^*$ and noting $\nabla \Phi (x^*) = 0$, we get:
\begin{align}\label{Lp-5}
\Phi (x) \leq& \Phi (x^*) + \frac{L_{\Phi}}{2} \| x - x^*\|^2.
\end{align}
This provides an upper bound:
\begin{align}\label{V_ndcondition1}
V_{nd} (t) = \Phi (x(t)) - \Phi (x^*(t)) \leq& \frac{L_{\Phi}}{2} \| x - x^*\|^2.
\end{align}

Define an auxiliary function as:
\begin{align}\label{Afunction}
\lambda(\theta) = \Phi ((1 - \theta)x +  \theta y ), \quad \forall x,y \in \mathbb R^{nm},
\end{align}
where $\theta$ is defined in Assumption 3, and $\lambda(\theta)$ is differentiable on the interval $\theta \in [0,1]$.
According to Assumption 3:
\begin{align}\label{Afunction1}
\lambda(\theta) \leq (1 - \theta) \Phi(x) + \theta \Phi(y) - \frac{m_{\Phi}}{2} \theta (1 - \theta) \|y - x \|^2.
\end{align}
Take the derivative of left-hand sides of (\ref{Afunction1}) with respect to $\theta$:
\begin{align}\label{dAfunction1-l}
\lambda'(\theta) = \nabla \Phi((1 - \theta)x + \theta y )^T (y - x).
\end{align}
At $\theta = 0$, we have:
\begin{align}\label{dAfunction1-l1}
\lambda'(0) = \nabla \Phi(x)^T (y - x).
\end{align}
For the right-hand side of (\ref{Afunction1}), taking the derivative gives:
\begin{align*}
\lambda'_r(\theta) =& - \Phi(x) + \Phi(y) - \frac{{m_{\Phi}}}{2} [(1-2\theta) \|y - x \|^2 - 2\theta(1 - \theta)\|y - x \|^2],
\end{align*}
where $\lambda_r(\theta) = (1 - \theta) \Phi(x) + \theta \Phi(y) - [{m_{\Phi}}\theta (1 - \theta) \|y - x \|^2]/{2} $.
At $\theta = 0$, this becomes:
\begin{align}\label{Afunction1-r}
\lambda'_r(0) =& - \Phi(x) + \Phi(y) - \frac{{m_{\Phi}}}{2}  \|y - x \|^2.
\end{align}
Since $\lambda(\theta) \leq \lambda_r(\theta) $, at $\theta = 0$, it follows: $\lambda'(0) \leq \lambda'_r(0) $.
By substituting (\ref{dAfunction1-l1}) and (\ref{Afunction1-r}), we have:
\begin{align*}
\nabla \Phi(x)^T (y - x) \leq - \Phi(x) + \Phi(y) - \frac{{m_{\Phi}}}{2}  \|y - x \|^2,
\end{align*}
which can further derives that:
\begin{align}\label{convex-condition}
\Phi(y) \geq \Phi(x) + \nabla \Phi(x)^T (y - x) + \frac{{m_{\Phi}}}{2}  \|y - x \|^2.
\end{align}
Similarly to (\ref{Lp-4})-(\ref{Lp-5}),
setting $x = x^*$ in equation (\ref{convex-condition}), the inequality transforms into the following form:
\begin{align}\label{convex-condition1}
\Phi (x) \geq & \Phi (x^*) + \frac{m_{\Phi}}{2} \| x - x^*\|^2.
\end{align}
By combining this result (\ref{convex-condition1}) with equation (\ref{V_ndcondition1}):
\begin{align}\label{V_ndcondition2}
\frac{m_{\Phi}}{2} \| x - x^*\|^2 \leq V_{1} (t) \leq \frac{L_{\Phi}}{2} \| x - x^*\|^2.
\end{align}

Let:
$$s_i = \vartheta \sum_{j = 1}^n a_{ij}(x_j (t) - x_i (t))  + \varepsilon \sum_{j \in {\mathcal N}_i} [\wp_{ij}(t) \textrm{sign}(x_j(t) - x_i(t))].$$
The derivative of (\ref{V_nd1}) along the trajectories of (\ref{closed-sys}) is given by:
\begin{align}\label{dV_nd1}
\dot V_{1} &= \frac{d}{dt}(\Phi (x(t)) - \Phi (x^*(t))) \notag\\
&= - k_{0} \sum_{i = 1}^n \eta_i^T(x_i (t)) \eta_i(x_i (t)) + \sum_{i = 1}^n \eta_i^T(x_i (t)) s_i(t) \notag\\
&\leq - {k_{0}} (\nabla^T \Phi(x) \nabla \Phi(x) ) + \sum_{i = 1}^n \eta_i^T(x_i (t)) s_i(t).
\end{align}
Based on Assumption 3 and (\ref{convex-condition}),
setting $y = x^*$,
we have:
\begin{align}\label{convex-condition2}
\Phi(x) - \Phi(x^*) \geq \nabla \Phi(x)^{\mathrm{T}} (x - x^*) + \frac{m_{\Phi}}{2} \| x - x^* \|^2.
\end{align}
By discarding the non-negative term ${m_{\Phi}} \| x - x^* \|^2/2$,
this reduces to:
\begin{align}\label{convex-condition3}
\Phi(x) - \Phi(x^*) \geq \nabla \Phi(x)^{\mathrm{T}} (x - x^*).
\end{align}
Using the Cauchy-Schwarz inequality, (\ref{convex-condition3}) can be further derived as:
\[
\Phi(x) - \Phi(x^*) \geq -\| \nabla \Phi(x) \| \| x - x^* \|.
\]

Taking the derivative of both sides of equation (\ref{convex-condition1}), we get:
\begin{align}\label{convex-condition4}
\| \nabla \Phi(x) \| \geq m_{\Phi} \| x - x^* \|.
\end{align}
Squaring both sides, it yields:
\[
\| \nabla \Phi(x) \|^2 \geq m_{\Phi}^2 \| x - x^* \|^2.
\]
Substituting $\| x - x^*\|^2 \geq {2}(\Phi (x) - \Phi (x^*))/{L_{\Phi}} $ from (\ref{Lp-5}),
one has:
\begin{align*}
\nabla^T \Phi(x) \nabla \Phi(x) \geq \kappa_{\Phi} (\Phi(x) - \Phi(x^*)),
\end{align*}
where $\kappa_{\Phi} = {2 m_{\Phi}^2}/{L_{\Phi}}$.
Substituting into equation (\ref{convex-condition3}) obtains:
\begin{align}\label{dV_nd2}
\dot V_{1} &\leq - k_{0} \kappa_{\Phi} (\Phi(x) - \Phi(x^*)) + \sum_{i = 1}^n \eta_i(x_i (t)) s_i(t)\notag\\
&\leq - \beta_{\Phi} V_{1} + \sum_{i = 1}^n \eta_i^T(x_i (t)) s_i(t),
\end{align}
where $\beta_{\Phi} = k_{0} \kappa_{\Phi} > 0$.

Since $e_{x_i}$ is globally exponential convergence,
$s_i(t)$ also converges exponentially, namely, $\| s(t) \| \leq \| s(0) \| e^{-\alpha_s t}$,
with $s(t) = [s_1(t),s_2(t),\ldots,s_n(t)]^T$ and $\alpha_s = 2\vartheta_1 \lambda_2(\bar {\mathcal L})$.
Based on Assumption 2, there exist a positive constant $\bar \eta$ satisfying that $\bar \eta \geq \| \nabla \Phi(x) \|$.
Using the Cauchy-Schwarz inequality,
\begin{align}\label{ISS-1}
\sum_{i = 1}^n \eta_i^T(x_i (t)) s_i(t) \leq& \| \nabla \Phi(x) \| \| s(t) \|  \leq \bar \eta  \| s(t) \|.
\end{align}
Substituting (\ref{ISS-1}) into (\ref{dV_nd2}) obtains:
\begin{align}\label{dV_nd3}
\dot V_{1} \leq& - \beta_{\Phi} V_{1} + \bar \eta  \| s(t) \| \notag\\
= & - \beta_{\Phi} V_{1} + \gamma_{s}(\| s(t) \|),
\end{align}
where $\gamma_{s}(\| s(t) \|) = \bar \eta  \| s(t) \|$ is obviously a class \( \mathcal{K} \) function.
From Lemma \ref{lemma7}, and the form of (\ref{V_ndcondition2}) and (\ref{dV_nd3}),
it can be concluded that the system (\ref{closed-sys}) is ISS with respect to input $s(t)$.
Using the ISS representation:
\begin{align*}
\|x(t) - x^*\| \leq \beta_{ISS}(\|x(0) - x^*\|,t ) + \gamma_{s} (\sup_{0 \leq \tau \leq t} \|s(\tau) \| ),
\end{align*}
and substituting $\| s(t) \| \leq \| s(0) \| e^{-\alpha_s t}$, $\forall t \geq 0$, we have:
\begin{align*}
\|x(t) - x^*\| \leq \beta_{ISS}(\|x(0) - x^*\|,t ) + \gamma_{s} (\| s(0) \| e^{-\alpha_s t}).
\end{align*}
Due to $\beta_{ISS} \in \mathcal{KL}$,
it satisfies that $\beta_{ISS}(\sigma,t ) \rightarrow 0$ when $t \rightarrow \infty$ for all $\sigma > 0$.
Additionally, $\gamma_{s} (\| s(0) \| e^{-\alpha_s t}) \rightarrow 0$ with $t \rightarrow \infty$.
Thus, it demonstrated that $\lim_{t \rightarrow \infty} \| x(t) - x^* \| = 0 $ as well as the positive definiteness and boundedness of $V_1$,
indicating that the system is globally asymptotically stable on $x^*$.

In the neighborhood of $x^*$, the influence of $s(t)$ for $x(t)$ tends to $0$,
implying that the input item $s(t)$ can be ignored.
Therefore, it yields:
\begin{align}\label{dV_nd4}
\dot V_{1} \leq& - \beta_{\Phi} V_{1}.
\end{align}
Using (\ref{dV_nd4}) and (\ref{V_ndcondition2}),
it follows that $x^*$, $x(t)$ exponentially converge to $x^*$ locally.
Thus, combining the analysis of (\ref{dV_nd1})-(\ref{dV_nd3}) and (\ref{dV_nd4}),
the closed-loop system (\ref{closed-sys}) is GALE stability at the optimal solution point $x^*$.

\textbf{\textrm{Step 2 (Stability with input delay):}}
Before proceeding, we clarify how the discontinuous sign-based channel is treated in the delay analysis. The SIDT algorithm \eqref{SIDT} contains the ideal sign function, and hence the resulting closed-loop vector field is generally not locally Lipschitz on the switching surfaces. Therefore, the following proof does not invoke the classical Razumikhin theorem by requiring the whole closed-loop functional to be locally Lipschitz. Instead, we use a Razumikhin-type comparison argument together with a decomposition of the feedback into a locally Lipschitz part and a uniformly bounded discontinuous part. More precisely, the delayed feedback mismatch is estimated by a Lipschitz term associated with the smooth channel and a bounded residual term induced by the sign-based channel. 

Building on the results of Lemma 8.1 in \cite{LinWei2020}, we establish that there exists a threshold \(\bar{d}_{c1}\) which depends on the initial domain radius \(r\). Specifically, for any delay \(d \in (0, \bar{d}_{c1}]\), all solution trajectories of the time-delay nonlinear system (\ref{system}), initialized with \( x_i(s) \) satisfying \( \sup_{s \in [-d, 0]} \| x_i(s) \| \leq r \), remain well-defined and do not experience finite escape time over the interval \([0, d]\). Furthermore, these trajectories are bounded by \( 2r \), where \( r \) is an arbitrarily large positive real number.

Using the Cauchy-Schwarz inequality,
consider a Lyapunov function candidate $V_{1} $ as (\ref{V_nd1}).
Define the level set $\Omega_{0,V_{1}} = \{x \in \mathbb R^n | V_{1}(x) \leq r_{0,V_{1}}\}$,
where $r_{0,V_{1}}$ is a positive constant satisfying:
$$r_{0,V_{1}} = \frac{L_{\Phi}}{2} \max(\| 2\sqrt{n}r - x^*\|, \| -2\sqrt{n}r - x^* \|)^2.$$
This result suggests that, for sufficiently small delays \( d > 0 \), the system's state $x_i(t)$, $\forall t \in [0, d]$ remains bounded and the solutions do not diverge within this interval.

Next, we extend the analysis to the time domain $t \in [d,+\infty) $.
By leveraging the dynamics of the closed-loop system:
\begin{align}\label{clossytem-td}
\dot x_{i} (t) = f_{i}(x_i(t)) + g_i (x_i(t)) \xi_i(x_i(t - d)),
\end{align}
it can be deduced from (\ref{dV_nd2}), the mean value theorem and Lemma \ref{lemma1} that for $t > d$, we derive that:
\begin{align}\label{dVnd-s21}
\dot V_{1} \leq  &
- a_1 \| x - x^*\|^2  + \left\| \frac{\partial V_{1}}{\partial x} \right\| \sum_{i = 1}^n \|g_i (x_i) \| \|\xi_i(x_i(t -d)) - \xi_i(x_i)   \| \notag\\
\leq  & - a_1 \| x - x^*\|^2 + \left\| \frac{\partial V_{1}}{\partial x} \right\| \|g (x) \| \| \xi(x(t -d)) - \xi(x) \|,
\end{align}
where $a_1 = \beta_{\Phi} m_{\Phi}/2$,
$g (x) = [g_1 (x_1), g_2 (x_2),\ldots,g_n (x_n)]^T$,
and $\xi = [\xi_1,\xi_2,\ldots,\xi_n]^T$.

Define an auxiliary error as $e_i = x_i - x_i^*$ and an auxiliary vector $e = [e_1,e_2,\ldots,e_n]^T$.
Split the control law as \(\xi(x)=\xi^{\mathrm{sm}}(x)+\xi^{\mathrm{sgn}}(x)\), where \(\xi^{\mathrm{sm}}(x)\) collects the smooth feedback terms and \(\xi^{\mathrm{sgn}}(x)\) denotes the sign-based consensus-enhancing term. On the compact set \(\Omega_{0,V_1}\), the smooth component \(\xi^{\mathrm{sm}}\) is locally Lipschitz, while the sign-based component \(\xi^{\mathrm{sgn}}\) is uniformly bounded but may be discontinuous on the switching surfaces. Therefore, we do not impose a Lipschitz estimate on \(\xi^{\mathrm{sgn}}\). Instead, its possible jump is absorbed into a bounded residual.
Since the solution trajectories remain in the compact set \(\Omega_{0,V_1}\), there exist positive constants \(L_{\xi,r}\), \(L_{\wp,r}\), \(\bar g_r\), and \(\bar \wp_r\) such that $\|g_i(x_i)^{-1}\|\le \bar g_r$, $|\wp_{ij}(x)-\wp_{ij}(y)|\le L_{\wp,r}\|x-y\|$ and $\wp_{ij}(x)\le \bar \wp_r$, for all \(x,y\in\Omega_{0,V_1}\).
From Lemmas \ref{lemma2} and \ref{lemma3}, the local Lipschitz property of \(\xi^{\mathrm{sm}}\), and the uniform boundedness of \(\xi^{\mathrm{sgn}}\) on \(\Omega_{0,V_1}\), there exist a smooth function \(\beta_1(x(t-d))\ge 1\) and constants \(b_1,b_2>0\), independent of \(d\), such that:
\begin{align}\label{S2-1}
\| \xi(x(t-d)) - \xi(x) \| \leq & b_1 \beta_1 (x(t-d)) \| e(t-d) - e(t)\| + b_2.
\end{align}
where $b_1=L_{\xi,r}+\varepsilon \bar g_r \Delta_G L_{\wp,r}$ and $b_2=2\varepsilon \bar g_r \Delta_G \bar\wp_r$ with $\Delta_G = \max_i | \mathcal N_i |$.

From (\ref{V_ndcondition2}) and (\ref{dV_nd2}), it follows that, for some $a_0 > 0$,
\begin{align}\label{S2-2}
\left. \left\| \frac{\partial V_{1}}{\partial x} \right\| \right|_{\Omega_{0,V_{1}}} \leq a_0 \| e \|.
\end{align}
Assume there exists a positive function $\beta_2 (x) \geq 1$ satisfying $g(x)\leq \beta_2 (x)$.
Thus, based on (\ref{dVnd-s21})-(\ref{S2-2}), for $t \geq d$, it obtains:
\begin{align}\label{dVnd-s22}
\left. \dot V_{1} (x) \right|_{\Omega_{0,V_{1}}} \leq& - a_1 \|e\|^2  + a_0\beta_2(x)\|e\| ( b_1 \beta_1 (x(t-d))  \|e(t-d)-e(t)\|+b_2)\notag\\
\leq& -\frac{a_1}{2}\|e\|^2 +\bar a_2 \beta_1^2 (x(t-d))\|e(t-d)-e(t)\|^2 +\bar a_3,
\end{align}
for some $\bar a_2,\bar a_3 > 0$ independent of the time delay $d$.
Here, \(\bar a_3\) is the residual term induced by the discontinuous sign-based channel and the boundedness of \(\beta_2(x)\) on \(\Omega_{0,V_1}\).

Let $\ell_i (x_i(t), u_i (t -d)) = f_i (x_i(t)) + g_i(x_i(t)) u_i (t -d)$.
Due to the smoothness of $f_i (x_i(t))$ and $g_i(x_i(t))$ with $f_i(\mathbf{0}_m) = \mathbf{0}_m$,
one has $\ell_i (\mathbf{0}_m, \mathbf{0}_m) = \mathbf{0}_m$.
Based on Lemmas \ref{lemma2} and \ref{lemma3},
it implies that there exist $\varpi_1 (x(\tau)) \geq 1$, $\varpi_2 (x(\tau - d)) \geq 1$, and a constant \(\varpi_0>0\), such that:
\begin{align}\label{dVnd-s23}
\left. \| e(t -d) - e(t)\|^2 \right|_{\Omega_{0,V_{1}}}
\leq& d \int_t^{t-d} \| \dot e(\tau) \|^2 d\tau \notag\\
\leq& d \int_{t-d}^t (\varpi_1 (x(\tau)) \| e (\tau)\|  + \varpi_2 (x(\tau - d)) \| e (\tau - d)\| \notag\\
& + \varpi_0)^2 d \tau \notag\\
\leq& 3 d \int_{t-d}^t (\varpi_1^2 (x(\tau)) \| e (\tau)\|^2 + \varpi_2^2 (x(\tau - d)) \| e (\tau - d)\|^2 \notag\\
&+ \varpi_0^2) d \tau \notag\\
\leq& 3 d^2 \sup_{\hbar \in [-2d,0]} \varpi_3 (x(t + \hbar)) e(t + \hbar) \|^2 + 3\varpi_0^2 d^2,
\end{align}
where $\ell = [\ell_1, \ell_2 , \ldots, \ell_n]^T$, and
$\varpi_3 (\bullet) = \varpi_1^2 (\bullet) + \varpi_2^2 (\bullet)$ is a smooth scalar function.
Substituting (\ref{dVnd-s23}) into (\ref{dVnd-s22}) obtains that, for all $t \in [d, +\infty)$,
\begin{align}\label{dVnd-s24}
\left. \dot V_{1} (x) \right|_{\Omega_{0,V_{1}}}
\leq& - \frac{a_1}{2} \|e\|^2 + 3 d^2 \bar a_2 \beta_1^2 (x(t - d)) \sup_{\hbar \in [-2d,0]} \varpi_3 (x(t + \hbar)) \| e(t + \hbar) \|^2 \notag\\
&+ \bar a_3 + 3\bar a_2 \beta_1^2 (x(t-d))\varpi_0^2 d^2.
\end{align}

Next, choose \(p(s)=3s\), which satisfies \(p(s)>s\) for all \(s>0\). 
In the following Razumikhin-type comparison, suppose that, for all \(\hbar\in[-2d,0]\),
\begin{align}\label{dVnd-s25}
V_{1}(x(t+\hbar)) < p[V_{1}(x(t))] = 3V_{1}(x(t)).
\end{align}
Then, using the quadratic bounds in \eqref{V_ndcondition2}, one obtains:
\begin{align}\label{dVnd-s26}
\frac{m_{\Phi}}{2} \| e(t + \hbar) \|^2
\leq&
V_{1} (x(t+ \hbar)) \notag\\
<&
3V_{1}(x(t)) \notag\\
\leq&
\frac{3L_{\Phi}}{2} \| e(t) \|^2,
\end{align}
for all \(t \in [d, +\infty)\) and \(\hbar \in [-2d,0]\).
Thus, one has:
\begin{align}\label{dVnd-s27}
\left. \sup_{\hbar \in [-2d,0]}  \| e(t + \hbar) \|^2 \right|_{\Omega_{0,V_{1}}} \leq \frac{3L_{\Phi}}{m_{\Phi}} \| e(t) \|^2,
\end{align}
and
\begin{align}\label{dVnd-s27b}
\left. \sup_{\hbar \in [-2d,0]}  \| e(t + \hbar) \| \right|_{\Omega_{0,V_{1}}} \leq& \sqrt{\frac{6}{m_{\Phi}}   \left. V_{1}(x(t)) \right|_{\Omega_{0,V_{1}}} } \notag\\
\leq& \sqrt{ \frac{6r_{0,V_{1}}}{m_{\Phi}}  }.
\end{align}

Substituting (\ref{dVnd-s27})-(\ref{dVnd-s27b}) into (\ref{dVnd-s24}),
it yields that for all $t \in [d, +\infty)$,
\begin{align}\label{dVnd-s28}
\left. \dot V_{1} (x) \right|_{\Omega_{0,V_{1}}}
\leq& - \frac{a_1}{2} \|e(t)\|^2 + 3 d^2 \bar a_2 a_4 a_5 \| e(t) \|^2 + \bar a_3 + 3 d^2 \bar a_2 a_6 \varpi_0^2,
\end{align}
where $a_4 = \sup_{\| \epsilon \| \leq  \sqrt{ \frac{6r_{0,V_{1}}}{m_{\Phi}}  }}
\beta_1^2 (\epsilon) \cdot \sup_{\| \epsilon \| \leq  \sqrt{ \frac{6r_{0,V_{1}}}{m_{\Phi}}  }}
\varpi_3 (\epsilon)$, $a_5 = \frac{3L_{\Phi}}{m_{\Phi}}$ and $a_6 = \sup_{\| \epsilon \| \leq  \sqrt{ \frac{6r_{0,V_{1}}}{m_{\Phi}}  }} \beta_1^2 (\epsilon)$.

Noting that \(a_4\), \(a_5\), and \(a_6\) are independent of $d$,
we provide a parameter $\bar a = 3\bar a_2 a_4 a_5$ and the maximal delay $\bar d_{c2} = \sqrt{ \frac{a_1}{4 \bar a} }$, depending on $r_{0,V_{1}}$ thereby on $r > 0$.
As a consequence, for $d \leq \bar d_{c2}$,
(\ref{dVnd-s28}) is reduced to:
\begin{align}\label{dVnd-s29}
\left. \dot V_{1} (x) \right|_{\Omega_{0,V_{1}}}  \leq  & - \frac{a_1}{4} \|e\|^2 + C_s(d),\quad \forall t \in [d, +\infty),
\end{align}
where $C_s(d)=\bar a_3 + 3 d^2 \bar a_2 a_6 \varpi_0^2$.
By \eqref{V_ndcondition2}, one has $\|e\|^2 \ge \frac{2}{L_{\Phi}}V_1(x)$.
Substituting this estimate into \eqref{dVnd-s29} gives, $\forall t \in [d, +\infty)$,
\begin{align}\label{dVnd-s30}
\left. \dot V_{1} (x) \right|_{\Omega_{0,V_{1}}}
\leq - \frac{a_1}{2L_{\Phi}} V_1(x) + C_s(d),
\end{align}

Let \(\bar d=\min\{\bar d_{c1},\bar d_{c2}\}>0\). Then, for any \(d\in(0,\bar d]\), the solution trajectories of the closed-loop system \eqref{clossytem-td}, starting from any continuous initial history \(x(s)\in C([-d,0],\mathbb R^n)\) with \(\|x(s)\|_C\le \sqrt n r\), are well defined and remain in the compact set \(\Omega_{0,V_1}\) for all \(t\ge0\). Moreover, the Razumikhin-type estimate \eqref{dVnd-s30} holds on \(\Omega_{0,V_1}\), where \(C_s(d)\) collects the bounded residual induced by the discontinuous sign-based channel and the delay-dependent mismatch.

Applying the comparison principle to \eqref{dVnd-s30}, one obtains:
\[
\limsup_{t\to\infty} V_1(x(t)) \le \frac{2L_{\Phi}}{a_1}C_s(d).
\]
Combining this estimate with \eqref{V_ndcondition2} further yields:
\begin{align}\label{dVnd-s31}
\limsup_{t\to\infty}\|x(t)-x^*\| \le \sqrt{\frac{4L_{\Phi}}{a_1m_{\Phi}}\,C_s(d)}.
\end{align}

Therefore, the closed-loop system \eqref{clossytem-td} is uniformly ultimately bounded with respect to the optimal solution \(x^*\). Hence, the NNS \eqref{system} exhibits practical IDTSC for distributed convex optimization as long as \(d\le \bar d\). In particular, the state \(x_i(t)\) semiglobally converges to a neighborhood of \(x_i^*\), whose size is determined by the residual term \(C_s(d)\). This residual originates from the bounded treatment of the discontinuous sign-based channel and the delay-dependent mismatch estimate, and thus the obtained ultimate bound is conservative.
\end{proof}

\begin{remark}
The present analysis is developed for a constant input delay.
As a potential extension, the time-varying Razumikhin stability theorem in \cite{ZhouB2016}
offers a possible route to treat bounded time-varying input delays within a similar proof architecture.
Specifically, one may attempt to adapt Step 2 of Theorem 1 by
(i) replacing fixed-delay shift estimates with bounds based on a delay envelope, and (ii) performing a Razumikhin comparison on a fixed window determined by the maximal delay.
A sketch of the required modifications is given below.

Consider a time-varying delay signal $d(\cdot)$ that is continuous (or piecewise continuous) and satisfies:
\[
0 \le d(t) \le \bar d_{\max}, \qquad \forall t\ge 0,
\]
where $\bar d_{\max}$ is a known delay envelope. In our semiglobal framework, one would further require
$\bar d_{\max}\le \bar d(r)$ so that the constant-delay admissibility margin remains respected.
In the constant-delay analysis, the delay enters through the standard estimate
$\|x(t)-x(t-d)\|\le d \sup_{h\in[-d,0]}\|\dot x(t+h)\|$.
For $d(t)$, the corresponding bound is:
\begin{align}\label{R2_1}
\|x(t)-x(t-d(t))\|
\le& \int_{t-d(t)}^{t}\|\dot x(\sigma)\|\,d\sigma \notag\\
\le& \bar d_{\max} \sup_{h\in[-\bar d_{\max},0]}\|\dot x(t+h)\|.
\end{align}
Thus, occurrences of $d$ and $\sup_{h\in[-d,0]}(\cdot)$ in Step 2 can be conservatively replaced by
$\bar d_{\max}$ and $\sup_{h\in[-\bar d_{\max},0]}(\cdot)$, respectively.

The Razumikhin comparison used in our proof (the $p(\cdot)$-argument) would be imposed on the fixed window $[-2\bar d_{\max},0]$:
\begin{align}\label{R2_2}
V(t+\theta)\le p \left(V(t)\right),\qquad \forall \theta\in[-2\bar d_{\max},0],
\end{align}
which yields a bound of the form
$\sup_{\theta\in[-2\bar d_{\max},0]}V(t+\theta)\le c_{\rm cmp}V(t)$ for some $c_{\rm cmp}>1$.

With (\ref{R2_1})-(\ref{R2_2}), the delayed closed loop can be rewritten in the standard functional form $\dot x(t)=f(t,x_t)$ on $C([-\bar d_{\max},0])$.
One may then apply the time-varying Razumikhin stability theorem (e.g., Theorem 1 in \cite{ZhouB2016})
to close the comparison argument and establish convergence under $d(t)$.
A complete extension would require re-deriving the admissibility condition and the comparison constants
for the time-varying delay case, which is challenging and is left for future work.
\end{remark}

\begin{remark}
This paper characterizes an explicit admissible input-delay margin $\bar d = \bar d(r)$ for the proposed SIDT algorithm \eqref{SIDT}, in the semiglobal sense that the tolerable delay depends on the prescribed initial radius $r$ of \eqref{system}.
In particular, $\bar d(r)$ is inversely proportional to $r$ that a larger admissible initial radius leads to a smaller delay margin.
Importantly, the controller implementation is delay-independent in that it does not require the knowledge of the realized delay value and does not involve delay-dependent gain tuning or structural modification,
and the same controller is applied for any admissible delay satisfying $d\le \bar d(r)$.
\end{remark}


\subsection{SIDT algorithm for distributed nonconvex optimization problem}
In this subsection, the application of the SIDT algorithm (\ref{SIDT}) for solving the distributed nonconvex optimization in delay-affected NNSs (\ref{system}) is investigated.

\begin{theoremx}\label{T1}
Suppose that Assumptions 1, 2, and 4 hold. For the distributed nonconvex optimization problem \eqref{Cx} implemented on the NNS \eqref{system}, if \(\varepsilon\ge 2n\), then the SIDT algorithm \eqref{SIDT} renders the closed-loop system input-delay tolerant in the practical semiglobal sense.
\end{theoremx}

\begin{proof}
The proof is similar to the proof of Theorem 1,
which can be divided into two steps.

\textbf{Step 1:} Demonstrate that the NNSs (\ref{system}) without input delay is global asymptotic and local exponential stability under the regulation of the SIDT algorithm (\ref{SIDT_t}).
Considering the non-delayed NNSs (\ref{system_nd}) and closed-loop system (\ref{closed-sys}),
similarly to the certification process of (\ref{V_t1})-(\ref{V_t14}),
it obtains that $e_{x_i}$ globally converges exponentially to origin.

Before moving on, according to \cite{Karimi2016},
the P-{\L} inequality implies that every stationary point is globally optimal.
Specifically, if $\nabla\Phi(y)=0$, then the P-{\L} inequality forces $\Phi(y)=\Phi^*$.
Consequently, non-optimal critical points, saddle points included, are excluded on any set where P-{\L} holds.
In this theorem, the optimizer is assumed to be unique, i.e., \(X^*=\{x^*\}\).

Next, we propose the Lyapunov function candidate:
\begin{align}\label{T2-V21}
V_{2} = \Phi (x(t)) - \Phi (x^*).
\end{align}
According to (\ref{dV_nd1}), the derivative of (\ref{T2-V21}) along the trajectories of (\ref{closed-sys}) yields that:
\begin{align}\label{T2-V22}
\dot V_{2} &\leq - {k_{0}}(\nabla^T \Phi(x) \nabla \Phi(x) ) + \sum_{i = 1}^n \eta_i^T(x_i (t)) s_i(t),
\end{align}
where
\begin{align*}
s_i(t) =&  \varepsilon \sum_{j \in {\mathcal N}_i} [\wp_{ij}(t) \textrm{sign}(x_j(t) - x_i(t))] + \vartheta \sum_{j = 1}^n a_{ij}(x_j (t) - x_i (t)).
\end{align*}
Based on (\ref{assumption3}) in Assumption 4,
we have:
\begin{align}\label{T2-a1}
\|\nabla \Phi(x) \|^2 \geq 2\mu_{\Phi} (\Phi(x) - \Phi(x^*)),
\end{align}
where $\mu_{\Phi} = \min_{1 \leq i \leq n} \mu_{\phi_i}$.
Substituting (\ref{T2-a1}) into (\ref{T2-V22}) obtains:
\begin{align}\label{T2-V23}
\dot V_{2} \leq& - {k_{0}} (\nabla^T \Phi(x) \nabla \Phi(x) ) + \sum_{i = 1}^n \eta_i^T(x_i (t)) s_i(t)\notag\\
\leq& - 2 {k_{0}} \mu_{\Phi} (\Phi(x) - \Phi(x^*)) + \sum_{i = 1}^n \eta_i^T(x_i (t)) s_i(t)\notag\\
\leq& - \beta_{\Phi,2} V_{2} + \sum_{i = 1}^n \eta_i^T(x_i (t)) s_i(t),
\end{align}
where $\beta_{\Phi,2} = 2 {k_{0}} \mu_{\Phi}> 0$.
Due to the globally exponential convergence of $e_{x_i}$,
it deduces that:
\begin{align}\label{T2-V24}
\dot V_{2} \leq& -  \beta_{\Phi,2} V_{2} + \| \nabla \Phi(x) \| \| s(t) \| \notag\\
\leq& -  \beta_{\Phi,2} V_{2} + \bar \eta  \| s(t) \| \notag\\
= & -  \beta_{\Phi,2} V_{2} + \gamma_{s}(\| s(t) \|),
\end{align}
where $\gamma_{s}(\| s(t) \|) = \bar \eta  \| s(t) \|$ is obviously a class \( \mathcal{K} \) function.

Similarly to the analysis of (\ref{dV_nd3})-(\ref{dV_nd4}),
it thus follows from (\ref{T2-V24}) that, in the vicinity of $x^*$,
the Lyapunov function $V_{2}$ satisfies the inequality:
\begin{align}\label{T2-V25}
\dot V_{2} \leq -  \beta_{\Phi,2} V_{2}.
\end{align}

Drawing upon the results from \cite{Karimi2016} and Lemma \ref{lemma5},
it can be deduced that any stationary point,
where $\nabla \Phi(x) = 0$,
is necessarily a global optimal solution.
Consequently, within the global domain of $x$,
it follows from (\ref{T2-V24})-(\ref{T2-V25}) that: $\lim_{t \rightarrow \infty} \| x(t) - x^* \| = 0 $.

Additionally, based on Assumption 4, $\Phi(x)$ is radial unbounded,
implying the Lyapunov function $V_{2}$ (\ref{T2-V21}) is also exhibits radial unboundedness.
Specifically, one has:
\begin{align*}
\lim_{\| x \| \rightarrow \infty} V_{2} = \infty.
\end{align*}
Therefore, it can be concluded that:
\begin{enumerate}
    \item[a)] $ V_{2}$ is both positive definite and radially unbounded.
    \item[b)] $\lim_{t \rightarrow \infty} \| x(t) - x^* \| = 0 $.
    \item[c)] in the neighborhood of $x^*$, $V_{2}$ satisfies $\dot V_{2} \leq -  \beta_{\Phi,2} V_{2}$.
\end{enumerate}
Thus, the closed-loop system \eqref{closed-sys} is GALE stable at \(x^*\).

\textbf{Step 2:}
Because the SIDT algorithm \eqref{SIDT} contains an ideal sign-based channel, the delayed closed-loop vector field is generally discontinuous on the switching surfaces. Therefore, the following argument is not based on the local Lipschitz continuity of the whole closed-loop functional. Instead, as in Theorem 1, we use a Razumikhin-type comparison estimate, where the smooth part of the feedback is treated through a local Lipschitz bound and the sign-based part is absorbed into a bounded residual.

From Step 1, the nominal closed loop admits, on the compact set \(\Omega_{0,V_2}\), the estimate:
\[
\nabla\Phi(x)^T[f(x)+g(x)\xi(x)]\le -\beta_{\Phi,2}V_2(x),
\]
possibly after reducing \(\beta_{\Phi,2}\).
In alignment with the analytical proof in Theorem 1,
consider the the delayed closed loop \eqref{clossytem-td}, differentiate $V_2$ along trajectories and add-subtract $\xi(x(t))$:
\begin{align}\label{T2-V26}
\dot V_2
= & \nabla\Phi(x(t))^T [f(x(t))+g(x(t)) \xi(x(t-d))]\notag\\
= & \underbrace{\nabla\Phi(x(t))^T [f(x(t))+g(x(t)) \xi(x(t))]}_{\le\,-\beta_{\Phi,2}V_{2}(t)\ \text{by }\eqref{T2-V25}} \notag\\
& +\nabla\Phi(x(t))^T g(x(t))(\xi(x(t-d))-\xi(x(t))).
\end{align}

Let $\Omega_{0,V_2}:=\{x\in\mathbb R^n: V_2(x)\le r_{0,V_2}\}$ be a compact sublevel set.
Since the closed-loop trajectories remain in \(\Omega_{0,V_2}\), the smooth part of \(\xi(\cdot)\) is Lipschitz on \(\Omega_{0,V_2}\), while the sign-based channel is uniformly bounded there.
Hence, there exist positive constants \(L_{\xi,r}>0\) and \(M_{\xi,r}>0\) such that
$\|\xi(x(t-d))-\xi(x(t))\| \le L_{\xi,r}\|x(t)-x(t-d)\|+M_{\xi,r}$, $\forall x(t),x(t-d)\in\Omega_{0,V_2}$.
Moreover, let $C_g = \sup_{x \in \Omega_{0,V_2}}\|g(x)\|$, $C_0 = C_g L_{\xi,r}$ and $C_1 = C_g M_{\xi,r}$.
Hence, from \eqref{T2-V26},
\begin{align}\label{T2-V27}
\left.\dot V_{2}(x)\right|_{\Omega_{0,V_2}} \le& -\beta_{\Phi,2}V_{2} + C_1 \|\nabla\Phi(x(t))\| + C_0 \|\nabla\Phi(x(t))\| \|x(t)-x(t-d)\|.
\end{align}

Using the standard trajectory estimate,
\begin{align}\label{T2-V28}
\|x(t)-x(t-d)\| \le d \sup_{h\in[-d,0]}\|\dot x(t+h)\|,
\end{align}
and noting that the delayed closed-loop vector field consists of a smooth part plus a bounded sign-based channel, there exist constants \(K_{\nabla,1}>0\) and \(K_{\nabla,2}>0\), independent of \(d\), such that:
\begin{align}\label{T2-V29}
\|\dot x(\tau)\|^2 \le K_{\nabla,1}\|\nabla\Phi(x(\tau))\|^2 + K_{\nabla,2},
\end{align}
whenever $x(\tau)\in\Omega_{0,V_2}$.

Combining \eqref{T2-V28}-\eqref{T2-V29} with \eqref{T2-V27} and applying Young's inequality with a tuned parameter yields:
\begin{align}\label{T2-V30}
\left.\dot V_{2}(x)\right|_{\Omega_{0,V_2}}
\le& -\frac{\beta_{\Phi,2}}{2} V_{2} -\frac{\beta_{\Phi,2}}{4\mu_\Phi} \|\nabla\Phi(x(t))\|^2 + C_1 \|\nabla\Phi(x(t))\|\notag\\
&+ C_0 \|\nabla\Phi(x(t))\| (d \sup_{h\in[-d,0]}\|\dot x(t+h)\|)\notag\\
\le& -\frac{\beta_{\Phi,2}}{2} V_{2} + \frac{2\mu_\Phi C_0^2}{\beta_{\Phi,2}} d^2 \sup_{h\in[-d,0]}\|\dot x(t+h)\|^2  + \frac{2\mu_\Phi C_1^2}{\beta_{\Phi,2}},
\end{align}
where the P-{\L} inequality $\|\nabla\Phi(x)\|^2\ge 2\mu_\Phi V_2(x)$ is utilized to cancel the mixed term.

Since \(\Omega_{0,V_2}\) is compact and \(\nabla\Phi(\cdot)\) is continuous, there exists a constant \(G_r>0\) such that:
\[
\|\nabla\Phi(x)\|\le G_r,\qquad \forall x\in\Omega_{0,V_2}.
\]
Invoking \eqref{T2-V29}, one has:
\begin{align}\label{T2-V31-new}
\left. \sup_{h\in[-d,0]}\|\dot x(t+h)\|^2 \right|_{\Omega_{0,V_2}}
\le K_{\nabla,1}G_r^2+K_{\nabla,2}.
\end{align}
Substituting \eqref{T2-V31-new} into \eqref{T2-V30} yields:
\begin{align}\label{T2-V33-new}
\left.\dot V_{2}(x(t)) \right|_{\Omega_{0,V_2}}
\le -\frac{\beta_{\Phi,2}}{2}V_2(x(t)) + C_{\mathrm{PL}}(d),
\end{align}
where
\[
C_{\mathrm{PL}}(d) =
\frac{2\mu_\Phi C_1^2}{\beta_{\Phi,2}} + \frac{2\mu_\Phi C_0^2 d^2}{\beta_{\Phi,2}} \left(K_{\nabla,1}G_r^2+K_{\nabla,2}\right).
\]

Choose the compact sublevel set \(\Omega_{0,V_2}\) such that it contains the prescribed initial history and satisfies:
\[
\sup_{\theta\in[-d,0]}V_2(x(\theta))<r_{0,V_2}.
\]
Moreover, by increasing \(r_{0,V_2}\) if necessary, assume that:
\[
r_{0,V_2}>\frac{2}{\beta_{\Phi,2}}C_{\mathrm{PL}}(d).
\]
Then, on the boundary \(V_2(x)=r_{0,V_2}\), it follows from
\eqref{T2-V33-new} that:
\begin{align}\label{T2-boundary}
\left.\dot V_2(x(t))\right|_{V_2=r_{0,V_2}} \le& 
-\frac{\beta_{\Phi,2}}{2}r_{0,V_2} + C_{\mathrm{PL}}(d) \notag\\
< &  0.
\end{align}
Therefore, the sublevel set \(\Omega_{0,V_2}\) is positively invariant for the delayed closed-loop system as long as the solution is well defined.

Together with the short-time well-posedness bound established in Theorem 1, there exists a delay bound:
\begin{align}\label{t2_bard}
    \bar d=\min\{\bar d_{c1},\bar d_{c2}\}>0,
\end{align}
depending on the prescribed initial radius \(r\), such that for every
\(d\in(0,\bar d]\), the solution of the delayed closed-loop system
\eqref{clossytem-td}, starting from any continuous initial history
\(x(\theta)\in C([-d,0],\mathbb R^n)\), is well defined and remains in
\(\Omega_{0,V_2}\) for all \(t\ge0\). Here, \(\bar d_{c1}\) guarantees the
short-time well-posedness on \([0,d]\), while \(\bar d_{c2}\) is chosen such
that the positive-invariance condition \eqref{T2-boundary} holds on
\(\Omega_{0,V_2}\).

Applying the comparison principle to \eqref{T2-V33-new}, for all \(t\ge d\),
one obtains:
\begin{align}\label{T2-V34-new}
V_2(x(t))
\le&
e^{-\frac{\beta_{\Phi,2}}{2}(t-d)}V_2(x(d)) \notag\\
&+
\frac{2}{\beta_{\Phi,2}}
\left(
1-e^{-\frac{\beta_{\Phi,2}}{2}(t-d)}
\right)
C_{\mathrm{PL}}(d).
\end{align}
Consequently,
\begin{align}\label{T2-V35-new}
\limsup_{t\to\infty}V_2(x(t))
\le
\frac{2}{\beta_{\Phi,2}}C_{\mathrm{PL}}(d).
\end{align}

Since the optimizer is unique, i.e., \(X^*=\{x^*\}\), and \(V_2(x)=\Phi(x)-\Phi(x^*)\)
is positive definite with respect to \(x^*\) on \(\Omega_{0,V_2}\), there exists
a class-\(\mathcal K\) function \(\alpha_{V_2}\) such that:
\begin{align}\label{T2-alpha}
\alpha_{V_2}(\|x-x^*\|)\le V_2(x),
\qquad \forall x\in\Omega_{0,V_2}.
\end{align}
Combining \eqref{T2-V35-new} with \eqref{T2-alpha} yields:
\begin{align}\label{T2-V36-new}
\limsup_{t\to\infty}\|x(t)-x^*\|
\le
\alpha_{V_2}^{-1}
\left(
\frac{2}{\beta_{\Phi,2}}C_{\mathrm{PL}}(d)
\right).
\end{align}

Therefore, under Assumption 4, the delayed closed-loop NNS
\eqref{clossytem-td} is input-delay tolerant in the practical semiglobal
sense for distributed nonconvex optimization. Specifically, for every
prescribed compact initial set, there exists an admissible delay bound
\(\bar d>0\) such that, for all \(d\in(0,\bar d]\), the state trajectory remains
well defined, stays bounded, and converges to a neighborhood of the unique
optimizer \(x^*\). The size of this ultimate neighborhood is characterized by
\(C_{\mathrm{PL}}(d)\), which collects the bounded residual induced by the
discontinuous sign-based channel and the delay-dependent mismatch estimate.
Thus, the obtained convergence result is practical rather than exact, and the ultimate bound is conservative.


\end{proof}

\begin{remark}
While the SIDT algorithm (\ref{SIDT}) demonstrates versatility in handling both types of optimization problems, its tolerance to input delays varies between the convex and non-convex conditions.
Specifically, the algorithm exhibits different levels of robustness to input delays (namely, the different upper bound delay $\bar d$ in (\ref{dVnd-s28})-(\ref{dVnd-s29}) and (\ref{t2_bard}), respectively) depending on the convexity properties of the optimization problem.
\end{remark}

\begin{remark}
The discontinuity of $\mathrm{sign}(\cdot)$ in (\ref{SIDT})-(\ref{SIDT_t}) may induce chattering.
A standard remedy is to replace $\mathrm{sign}(s)$, for all $s\in\mathbb R^m$, by a continuous boundary-layer approximation:
\begin{align*}
\operatorname{sat}_\sigma (s) = \frac{s}{\max\{\sigma,  \|s\|\}},\quad \sigma>0.
\end{align*}
These proxies satisfy that $\|\operatorname{sat}_\sigma(s)\|\le 1$, $s^T \operatorname{sat}_\sigma(s) \ge \|s\|-\sigma$ and $\operatorname{sat}_\sigma (s)\to s/ \| s \|$ pointwise for $s\neq 0$ as $\sigma \searrow 0$.
It is noted that our proof in Theorems 1 and 2 use only boundedness and this sector inequality.
Therefore, the analysis carries over verbatim with $\mathrm{sign}$ replaced by $\operatorname{sat}_\sigma$.
In particular, for Theorems 1 and 2, there exist constants $c_1,c_2>0$ (independent of $\sigma$) and a delay bound $\bar d>0$ such that, for all $d\le \bar d$,
\begin{align*}
\dot V_l \le -c_1 V_l + c_2\sigma,
\end{align*}
whence
\begin{align*}
V_l  \le e^{-c_1(t-t_0)} V_l(t_0) + \frac{c_2}{c_1}\sigma,
\end{align*}
where $l \in \{1,2\}$.
Namely, practical consensus in the Lyapunov metric within an $\mathcal O(\sigma)$ neighborhood.
\end{remark}

\begin{remark}
The SIDT algorithm gains $k_0$, $\vartheta$, $\varepsilon$ in (\ref{SIDT}) is employed to obtain explicit admissible delay margins $\bar d$ in Theorems 1 and 2.
As usual, there is a speed-delay trade-off such that increasing the gains accelerates convergence but also enlarges the constants $\bar a$ in (\ref{dVnd-s28}) and $\beta_{\Phi,2}$ in (\ref{T2-V33-new})-\eqref{T2-V34-new} that ultimately shrink $\bar d$.
Moreover, the condition $\varepsilon\ge 2n$ in Theorems 1 and 2 is a uniform, topology-agnostic bound ensuring the consensus term dominates aggregated disagreement,
which is conservative and scales with network size.
In large networks this pushes the minimal admissible $\varepsilon$ upward and thereby reduces the delay tolerance.
Hence, in practice one should balance network size and sparsity against the required delay robustness when selecting gains.
Furthermore, a systematic treatment of adaptive/scheduled gains and topology-dependent bounds that avoid explicit $n$-scaling is a promising direction for future work.
\end{remark}



\section{Numerical Examples}

\begin{figure}[t]
\centering
\makebox{\includegraphics[width=6cm]{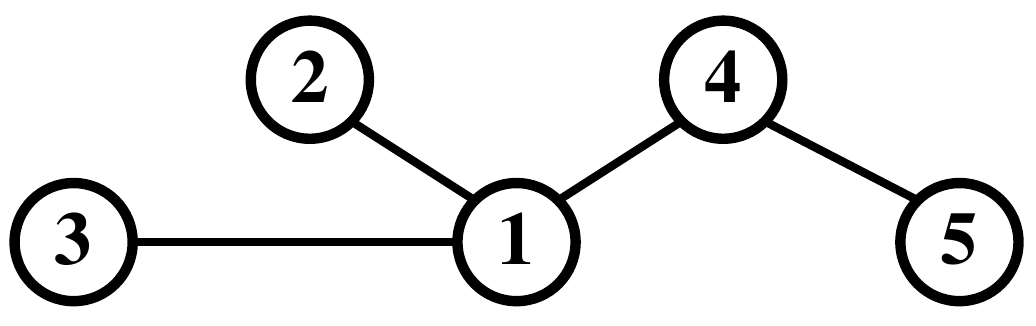}} 
\caption{Illustration of communication network in the NNS (\ref{system}).}\label{top}
\end{figure}

In this section, we present two numerical examples to demonstrate the proposed SIDT algorithm (\ref{SIDT}) within the frameworks of convex and nonconvex optimization, respectively.
Both examples are grounded in the identical NNS (\ref{system}) and utilize the same control parameters in (\ref{SIDT}).
We begin by outlining the shared NNS's configuration and control parameter settings that are fundamental to both optimization scenarios.
As shown in Fig. \ref{top}, consider a NNS (\ref{system}) consisting of $5$ nodes communicating under the undirected graph ${\mathcal{G}}$ such that:
\begin{align}\label{NE-G-L}
{\mathcal L} = \begin{bmatrix}
3 & -1 & -1 & -1 & 0 \\
-1 & 2 & -1 & 0 & 0 \\
-1 & -1 & 2 & 0 & 0 \\
-1 & 0 & 0 & 2 & -1 \\
0 & 0 & 0 & -1 & 1 \\
\end{bmatrix},
\end{align}
where ${\mathcal L}$ is the Laplacian matrix of the graph ${\mathcal{G}}$.
The dynamical model of the NNS (\ref{system}) with $m = 1$ is presented as:
\begin{align}\label{NE-system}
\dot{x}_i(t) =& x_i^2(t) + u_i(t-d), \quad i = 1, \ldots, n,
\end{align}
where $d$ is the input delay.
For each example, we determine the admissible delay threshold $\bar d$ separately,
since the objective's curvature properties (e.g., strong convexity and P-{\L} conditions) affect the comparison constants that define $\bar d$.
It is important to emphasize that although we set \( d \), the controller does not require explicit knowledge of $d$ during simulations.
Instead, it operates solely based on the system state affected by the delay \( d \).
Then, the setting of control parameters in (\ref{SIDT}) is proposed as:
$\vartheta = 0.1$, $\varepsilon = 11$, and $k_0 = 0.1$.
To proceed, two numerical examples are presented as follows.

\begin{table}[t]
\caption{The local convex cost function of NNS (\ref{NE-system})}
\label{NE-Table1}
\centering
\renewcommand\arraystretch{1.5}{
\setlength{\tabcolsep}{9mm}{
\begin{tabular}{p{1.5cm}<{\centering} p{3cm}<{\centering}}
\hline
\textbf{Node index} & \textbf{Local cost function}     \\ \hline
\textbf{1}           & $\phi_1 = x_1^2 + x_1 + 1$          \\
\textbf{2}           & $\phi_2 = x_2^2 + 2x_2 + 2$          \\
\textbf{3}           & $\phi_3 = x_3^2 + 3x_3 + 3$          \\
\textbf{4}           & $\phi_4 = x_4^2 + 4x_4 + 4$          \\
\textbf{5}           & $\phi_5 = x_5^2 + 5x_5 + 5$           \\ \hline
\end{tabular}}}
\end{table}

\begin{figure}[t]
\centering
\makebox{\includegraphics[width=8.2cm]{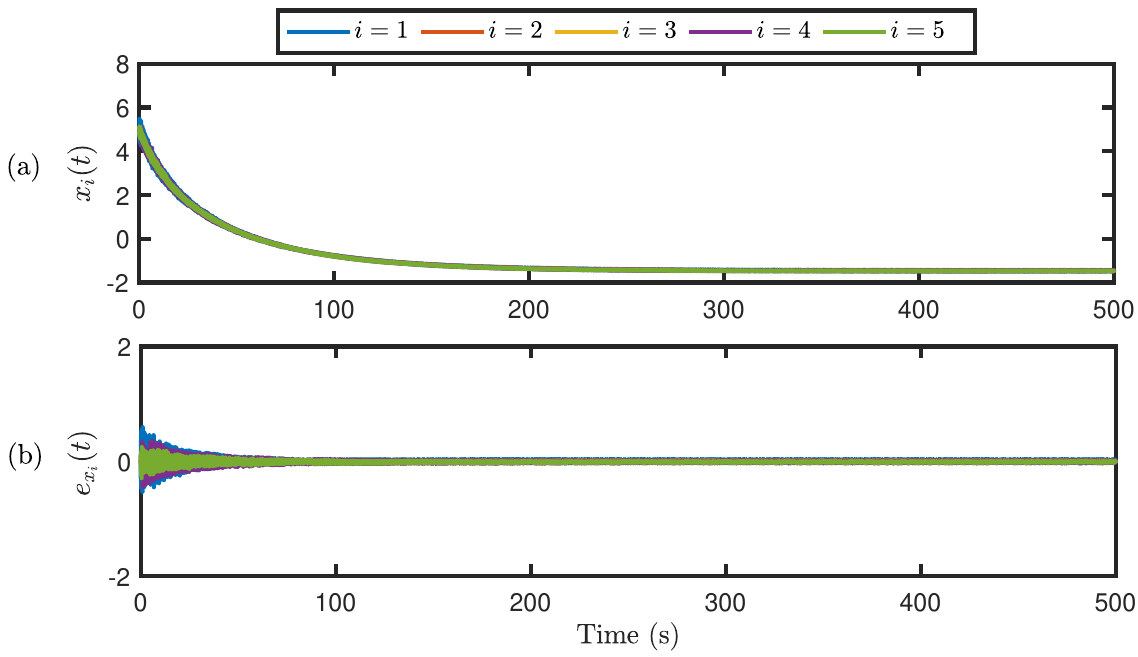}} 
\caption{Evolution of system states \( x_i(t) \) and the state errors \( e_{x_i}(t) \).
(a): Evolution of \( x_i(t) \), $\forall i \in {\mathcal V}$.
(b): Evolution of \( e_{x_i}(t) \), $\forall i \in {\mathcal V}$.}\label{CO-xi}
\end{figure}

\begin{figure}[t]
\centering
\makebox{\includegraphics[width=8.5cm]{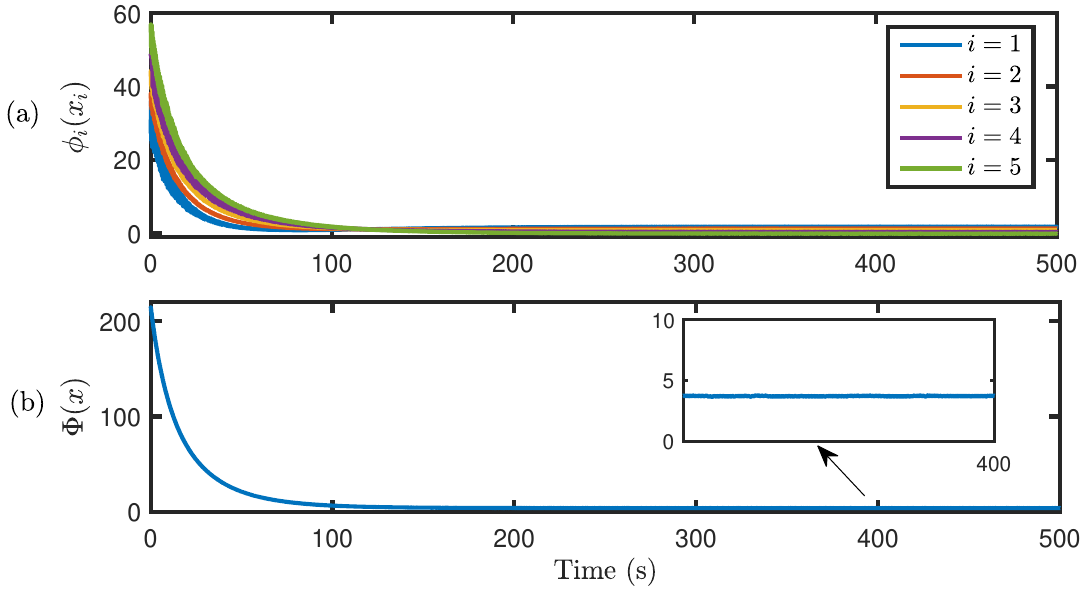}} 
\caption{Evolution of local and global cost functions.
(a): Evolution of $\phi_i (x_i)$, $\forall i \in {\mathcal V}$.
(b): Evolution of $\Phi(x)$. }\label{CO-f}
\end{figure}

\begin{example}
\textbf{(Example for distributed convex optimization)}
In this numerical example,
the SIDT algorithm (\ref{SIDT}) is utilized for the distributed convex optimization in the delay-affected NNS (\ref{NE-system}).
The local cost functions are respectively provided in Table \ref{NE-Table1}.
We set the initial radius \( r = 5 \), namely, initial at $(- 5, 5)$, and select the delay parameter \( d \) randomly in the range \( \{0.01, 0.03, 0.05\} \).

The simulation results of the numerical examples are presented in Figs. \ref{CO-xi} and \ref{CO-f}.
By Fig. \ref{CO-xi}, it illustrates the temporal evolution of the system states \( x_i(t) \) and the state errors \( e_{x_i}(t) \).
It is evident that the states converge to consensus, thereby satisfying the constraints of the distributed optimization problem (\ref{Cx}), and practically achieving the optimal solution \( x^* =  - 1.47 \).
As shown in Fig. \ref{CO-f},
it depicts the dynamics of both the local cost functions $\phi_i (x_i)$, $\forall i \in {\mathcal V}$, and global cost functions $\Phi(x)$ over time.
Based on these results, it can be concluded that the SIDT algorithm (\ref{SIDT}) effectively regulates delay-affected NNS (\ref{NE-system}) to attain IDTSC in the practical sense for distributed optimization.
\end{example}

\begin{table}[htbp]
\caption{The local nonconvex cost function of NNS (\ref{NE-system})}
\label{NE-Table2}
\centering
\renewcommand\arraystretch{1.5}{
\setlength{\tabcolsep}{9mm}{
\begin{tabular}{p{1.5cm}<{\centering} p{5cm}<{\centering}}
\hline
\textbf{Node index} & \textbf{Local cost function}          \\ \hline
\textbf{1}           & $\phi_1 = (x_1 - 1)^2 + 2\sin^2(x_1 - 1)$          \\
\textbf{2}           & $\phi_2 = (x_2 - 2)^2 + 2\sin^2(x_2 - 2)$         \\
\textbf{3}           & $\phi_3 = (x_3 - 1)^2 + 2\sin^2(x_3 - 1)$          \\
\textbf{4}           & $\phi_4 = (x_4 - 2)^2 + 2\sin^2(x_4 - 2)$          \\
\textbf{5}           & $\phi_5 = (x_5 - 1)^2 + 2\sin^2(x_5 - 1)$          \\ \hline
\end{tabular}}}
\end{table}

\begin{figure}[t]
\centering
\makebox{\includegraphics[width=8cm]{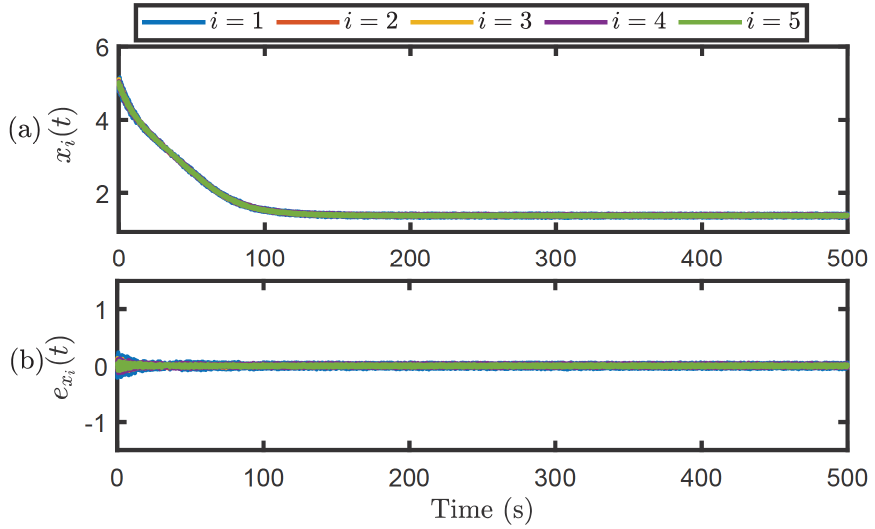}} 
\caption{Evolution of system states \( x_i(t) \) and the state errors \( e_{x_i}(t) \).
(a): Evolution of \( x_i(t) \), $\forall i \in {\mathcal V}$.
(b): Evolution of \( e_{x_i}(t) \), $\forall i \in {\mathcal V}$.}\label{NCO-xi}
\end{figure}

\begin{figure}[t]
\centering
\makebox{\includegraphics[width=8cm]{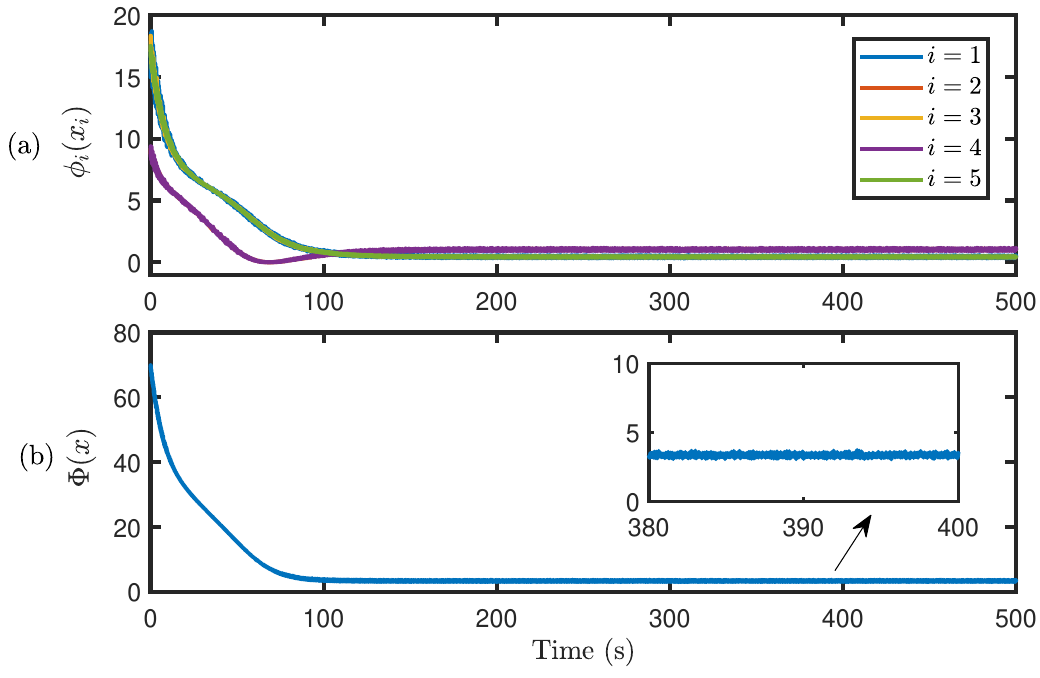}} 
\caption{Evolution of local and global cost functions.
(a): Evolution of $\phi_i (x_i)$, $\forall i \in {\mathcal V}$.
(b): Evolution of $\Phi(x)$. }\label{NCO-f}
\end{figure}

\begin{example}
\textbf{(Example for distributed nonconvex optimization)}
In this numerical example, the SIDT algorithm (\ref{SIDT}) is applied to a distributed nonconvex optimization problem within a delay-affected NNS (\ref{NE-system}).
The local and global cost functions for this scenario are detailed in Table \ref{NE-Table2}.
We initialize the system with a radius \( r = 5 \), namely, initial at $(- 5, 5)$,
and select the delay parameter \( d \) randomly in the range \( d \in \{0.01, 0.02, 0.03\} \).

The simulation results are shown in Fig.s \ref{NCO-xi} and \ref{NCO-f}.
Fig. \ref{NCO-xi} shows the temporal evolution of the system states \( x_i(t) \) and state errors \( e_{x_i}(t) \).
It clearly demonstrates that the states $x_i (t)$ converge to consensus, thereby satisfying the consensus constraints and practically converging the optimal solution \( x^* = 1.4 \).
In Fig. \ref{NCO-f}, it presents the evolution of the local cost functions \( \phi_i(x_i) \), for all \( i \in \mathcal{V} \), and the global cost function \( \Phi(x) \) over time.
These phenomena substantiate that the SIDT algorithm (\ref{SIDT}) effectively controls the delay-affected NNS (\ref{NE-system}), achieving IDTSC in the practical sense for distributed nonconvex optimization.
\end{example}

\section{Conclusion}
In this paper, we have introduced the IDTSC concept, providing a robust foundation for analyzing distributed optimization problems in NNSs with input delays and consensus constraints.
To achieve IDTSC practically in constrained optimization problem of delay-affected NNSs,
a novel SIDT algorithm has been proposed.
Furthermore, it has been demonstrated that the proposed algorithm can be extended to nonconvex optimization problems under P-{\L} condition, showcasing the algorithm's robustness and adaptability in less restrictive optimization environments.
Comprehensive simulations have validated the efficacy of the proposed algorithm, affirming its capability to achieve IDTSC in both convex and nonconvex distributed optimization scenarios within delay-affected NNSs.
The future work is to explore the optimization in NNSs subject to time-varying input delays and affine formation constraints.

\end{document}